\documentclass[10pt]{amsart}
\usepackage{amsfonts}
\usepackage{amssymb}
\usepackage{amsmath}
\usepackage{amsthm}
\usepackage{latexsym}
\usepackage{color}
\usepackage[normalem]{ulem}

\usepackage{amscd}
\usepackage[all,cmtip]{xy}
\usepackage{mathrsfs}
\usepackage{graphicx}
\usepackage[textheight=21cm, textwidth=13cm]{geometry}
\usepackage{esint}
\usepackage[usenames,dvipsnames,svgnames,table]{xcolor}
\usepackage[colorlinks=true,
pdfstartview=FitV,linkcolor=ForestGreen,citecolor=ForestGreen,
urlcolor=blue]{hyperref}
\usepackage[capitalise]{cleveref}
\allowdisplaybreaks
\numberwithin{equation}{section}
\usepackage{tikz}
\usetikzlibrary{backgrounds}
\usepackage{comment}
\usepackage{soul}

\newtheorem{theorem}{Theorem}

\newtheorem{lemma}[theorem]{Lemma}

\theoremstyle{definition}
\newtheorem{definition}[theorem]{Definition}

\theoremstyle{remark}
\newtheorem{remark}[theorem]{Remark}
\newtheorem{remarks}[theorem]{Remarks}

\newcommand\N{{\mathbb N}}
\newcommand\R{{\mathbb R}}

\newcommand{\I}{\text{Id}}
\newcommand{\U}{\text{{\bf 1}}}
\newcommand{\V}{\text{{\bf 0}}}

\newcommand\ds{{\frac{\mathrm d}{\mathrm ds}}}

\newcommand{\dd}{{\, \mathrm d}}

\newcommand{\var}{\varepsilon}
\newcommand{\Id}{\hbox{Id}}

\let\oldmarginpar\marginpar
\renewcommand\marginpar[1]{\-\oldmarginpar[\raggedleft\footnotesize #1]%
{\raggedright\footnotesize #1}}

\def \ell {{\mathit{l}}}

\newcommand{\OO}{\mathbb{O}}

\def \s {{\sigma}}
\def\p{\partial}

\newcommand{\sB}{{\mathsf B}}
\newcommand{\cA}{{\mathcal A}}
\newcommand{\cB}{{\mathcal B}}

\newcommand{\cI}{{\mathcal I}}
\newcommand{\cJ}{{\mathcal J}}

\newcommand{\cT}{{\mathcal T}}
\newcommand{\cO}{{\mathcal O}}
\newcommand{\sW}{{\mathsf W}}

\newcommand{\sT}{{\mathsf T}}

\renewcommand{\vec}[1]{\mathbf{#1}}

\title[Hypoelliptic Poincaré inequality and De Giorgi
methods]{Poincaré inequality and Quantitative De Giorgi method
  for hypoelliptic operators}

\author[F. Anceschi]{Francesca Anceschi}

\address[Francesca Anceschi]{Università Politecnica delle Marche:
  Dipartimento di Ingegneria Industriale e Scienze Matematiche,
  Via Brecce Bianche 12 60131 Ancona, Italy}
\email{f.anceschi@staff.univpm.it}

\author[H. Dietert]{Helge Dietert}

\address[Helge Dietert]{Université Paris Cité and Sorbonne
  Université, CNRS, IMJ-PRG, F-75006 Paris, France}
\email{helge.dietert@imj-prg.fr}

\author[J. Guerand]{Jessica Guerand}

\address[Jessica Guerand]{Université de Montpellier, 499-554 Rue
  du Truel, 34090 Montpellier, France}
\email{jessica.guerand@umontpellier.fr}

\author[A. Loher]{Amélie Loher}

\address[Amélie Loher]{University of Cambridge, Wilberforce Road,
  Cambridge CB3 0WA, United Kingdom} \email{ajl221@cam.ac.uk}

\author[C. Mouhot]{Cl\'ement Mouhot}

\address[Cl\'ement Mouhot]{University of Cambridge, Wilberforce
  Road, Cambridge CB3 0WA, United Kingdom}
\email{c.mouhot@dpmms.cam.ac.uk}

\author[A. Rebucci]{Annalaura Rebucci}

\address[Annalaura Rebucci]{Max Planck Institute for Mathematics
  in the Sciences, Inselstraße 22, 04103 Leipzig, Germany}
\email{annalaura.rebucci@mis.mpg.de}

\thanks{AL is supported by the Cambridge Trust.}

\date{Apr 22, 2026}

\begin{document}

\begin{abstract}
  We propose a systematic approach based on
  trajectories to prove a Poincaré inequality for weak non-negative sub-solutions to hypoelliptic
  equations with an arbitrary number of Hörmander commutators,
  both in the local and in the non-local case. 
  As a consequence, we deduce the weak Harnack inequality and H\"older regularity along the line of the De Giorgi method. 
\end{abstract}

\maketitle

\tableofcontents

\section{Introduction}
\label{sec:intro}

\subsection{The question at hand}

This paper is devoted to the regularity theory of hypoelliptic
equations with rough coefficients. We consider equations with an
arbitrary number of commutators in the sense of Hörmander~\cite{MR0222474}, 
with a diffusive part that is either local, in which case the equation is of second order, 
or non-local, in which case we deal with an equation of fractional order. In either case though, 
the diffusion is \textit{degenerate}, as it acts only along the direction of one vector field.
Our aim is to recover regularity in all directions, by transferring regularity from the degenerate diffusivity onto 
all vector fields.

Concretely, we consider an unknown function $f=f(x,t)$ solving an equation of the form
\begin{equation}
  \label{eq:hormander}
  \partial_t f + \cB f + \cA^* \cA f =0
\end{equation}
where $\cB$ is a first-order differential operator and
$\cA$ is a, possibly fractional, differential operator. 
Here $t \in \R$  and $x=(x^{(\kappa)},\dots,x^{(1)},x^{(0)})$ 
with each
$x^{(i)} \in \R^{d_i}$, $d_i \in \N^*$ and
$i \in \{0,\dots,\kappa\}$ such that
$N:=d_0+\cdots+d_\kappa$. The differential operator $\cA$ acts merely in one direction of $x$, namely $x^{(0)}$,
which makes up the diffusive part of the equation.
Even though we assume \textit{no regularity} on the coefficients of $\cA$, we still expect the diffusivity that stems from $\cA$ to regularise the
solution $f$ of \eqref{eq:hormander} at least in the variable $x^{(0)}$. To recover a regularisation in all remaining directions of $x$, 
we require a so-called \textit{Hörmander's commutator condition}: the constant coefficient vector fields, and their commutators 
(defined as $[X, Y] = XY - YX$, where $X$ and $Y$ are any two smooth vector fields) appearing in \eqref{eq:hormander} are supposed 
to span the whole space. Any equation with this property is called \textit{hypoelliptic}. In our notation, $\kappa$ denotes the number of commutators. 

The aim of the current paper is to derive a Poincaré inequality for non-negative weak sub-solutions to \eqref{eq:hormander}, 
which exploits Hörmander's hypoellipticity to deal with the degeneracy in the diffusive part:
\begin{equation}\label{eq:aim-ineq}
  \int_{Q_+} \left( f - \frac{1}{|Q_-|} \int_{Q_-} f
    \right)_+ \lesssim \int_\Omega |\cA f|,
\end{equation}
where $Q^-,Q^+$ are two disjoint cylinders in the open domain
$\Omega \subseteq \mathbb{R}^{N+1}$ so that $Q^+$ is a time-translation of $Q^-$ into the
future. The symbol $\lesssim$ denotes that the inequality holds up to a universal constant $C$. 
The striking feature of \eqref{eq:aim-ineq} is that on the right hand side there appears merely a differential in the diffusive direction: it is crucial that no differential in other directions appears in order to deduce the weak Harnack inequality or Hölder regularity of solutions to \eqref{eq:hormander}.

In this paper we prove a Poincaré inequality with a
quantitative control on the constant hidden in the symbol $\lesssim$ in \eqref{eq:aim-ineq} for a large class of
equations of type~\eqref{eq:hormander}. 
The Poincaré inequality is a
powerful integral way to measure the oscillation of the
sub-solution, and its use in this context goes back
to~\cite{nash,moser}. We use it to quantitatively to deduce the De Giorgi
lemmas (and thus the Hölder regularity for weak solutions), and the (weak)
Harnack inequality for super-solutions following the methodology
recently developed
in~\cite{guerand-mouhot,loher-2022-preprint-quantitative-de-giorgi}. The main novelty is
the method itself, but the result is
also new when $\mathcal A$ is a fractional differential operator in the case that more than one commutator is
involved (i.e. $\kappa > 1$).

The proof is based on a systematic construction of trajectories that encode the flow of the vector fields underlying \eqref{eq:hormander}, in such a way that they relate the future cylinder $Q^+$ to any point in the past cylinder $Q^-$. 
The use of trajectories is inspired from \cite{guerand-mouhot}, where they consider the local case (second order equation) with one commutator $\kappa = 1$. Their trajectories were combining the vector fields of the equation in a piecewise affine way; yet this was sub-optimal and they had an error term on the right hand side of \eqref{eq:aim-ineq}. This approach was improved by L. Niebel and R. Zacher in \cite{niebel2022kinetic}, where they construct an ad-hoc smooth curve instead of a piecewise affine curve; however this was not sufficient to get rid of the error term that also appeared in \cite{guerand-mouhot}. Here we introduce an approach which picks up on ideas of both these papers: we construct a \textit{smooth} trajectory as the solution of a control problem associated to the equation \eqref{eq:hormander}. We end up with a clean Poincaré inequality in the form \eqref{eq:aim-ineq} without any error terms on the right hand side.

\subsection{Examples of equations}

When $\kappa=1$, the prototypical local equation is the so-called
\emph{Kolmogorov equation}
\begin{equation}
  \label{eq:kolmo}
  \partial_t f + v \cdot \nabla_x f = \nabla_v \cdot \left[ A
    \nabla_v f \right]
\end{equation}
on the unknown $f=f(x,v,t)$ depending on $x,v \in \R^d$ with
$d \in \N$, $t \in \R$, where the matrix $A=A(x,v,t)$ is
measurable symmetric so that, for some $\Lambda \ge 1$,
\begin{equation}
  \label{eq:ellip-1}
  \forall \, (x,v,t) \in \R^{2d+1}, \quad \Lambda^{-1} \le
  A(x,v,t) \le \Lambda.
\end{equation}

When $\kappa \ge 2$, the underlying geometry becomes more involuted, and 
the prototypical local equation is
\begin{equation}
  \label{eq:kolmo-k}
  \partial_t f + x^{(\kappa-1)} \cdot \nabla_{x^{(\kappa)}} f
  + x^{(\kappa-2)} \cdot \nabla_{x^{(\kappa-1)}} f +
  \dots + v \cdot \nabla_{x^{(1)}} f =
  \nabla_v \cdot \left[ A \nabla_v f \right]
\end{equation}
on the unknown
$f=f(x^{(\kappa)},x^{(\kappa-1)}, \dots, x^{(1)},v,t)$, depending
on $x^{(\kappa)},x^{(\kappa-1)}, \dots, x^{(1)},v \in \R^d$,
$t \in \R$, where the matrix $A$ is measurable, symmetric and so
that, for some $\Lambda \ge 1$,
\begin{equation} \label{eq:ellip-k}
\begin{aligned}
  \forall \, (x^{(\kappa)},x^{(\kappa-1)}, \dots, x^{(1)},v,t)
 & \in \R^{(\kappa+1)d+1}, \\
   &\Lambda^{-1} \le
  A(x^{(\kappa)},x^{(\kappa-1)},\dots,x^{(1)},v,t) \le \Lambda.
 \end{aligned}
\end{equation}
From now on, we denote $v = x^{(0)}$ for notational consistency.

Our method can also handle fractional (non-local) operators.  We demonstrate the
application to the fractional setting by considering a fractional toy problem
with rough coefficients.  In the case \(\kappa=1\), it takes the form
\begin{equation}
  \label{eq:kolmo-nonloc}
  \partial_t f + v \cdot \nabla_x f = \left( - \Delta_v
  \right)^{\frac{\beta}{2}} \cdot \left[ a
    \left( - \Delta_v \right)^{\frac{\beta}{2}} f \right]
\end{equation}
on the unknown $f=f(x,v,t)$ depending on $x,v \in \R^d$ and
$t \in \R$, with $\beta \in (0,1)$ and a scalar function $a$ that
is measurable and so that, for some $\Lambda \ge 1$,
\begin{equation}
  \label{eq:ellip1-frac}
  a(x,v,t) \in [\Lambda^{-1},\Lambda].
\end{equation}
For a higher number of commutators, the corresponding version is
\begin{equation}
  \label{eq:kolmo-nonloc-k}
 \partial_t f + x^{(\kappa-1)} \cdot \nabla_{x^{(\kappa)}} f
  + x^{(\kappa-2)} \cdot \nabla_{x^{(\kappa-1)}} f +
  \dots + v \cdot \nabla_{x^{(1)}} f =
  \left( - \Delta_v \right)^{\frac{\beta}{2}} \cdot \left[ a
    \left( - \Delta_v \right)^{\frac{\beta}{2}} f \right]
\end{equation}
on the unknown $f$ which depends on
$x^{(\kappa)},x^{(\kappa-1)}, \dots, x^{(1)},v \in \R^d$,
$t \in \R$, and where the function $a$ is measurable and so that,
for some $\Lambda \ge 1$,
\begin{equation}
\begin{aligned}
  \label{eq:ellipk}
  \forall \, (x^{(\kappa)},x^{(\kappa-1)}, \dots, x^{(1)},v,t)
  &\in \R^{(\kappa+1)d+1}, \\
   &\Lambda^{-1} \le
  a(x^{(\kappa)},x^{(\kappa-1)},\dots,x^{(1)},v,t) \le \Lambda.
  \end{aligned}
\end{equation}

The idea of this fractional setting is to have a divergence form equation of
fractional order $2\beta \in (0, 2)$ with rough diffusion coefficients.

Such a setting is for example motivated by the Boltzmann equation (without
cutoff), which is a kinetic equation with a degenerate fractional diffusion
operator in velocity.  In this application, the nonlocal operator
$\left( - \Delta_v \right)^{\frac{\beta}{2}} \cdot \left[ a \left( - \Delta_v \right)^{\frac{\beta}{2}} f \right]$
in \eqref{eq:kolmo-nonloc} is replaced by
\begin{align*}
  & \mathcal L f\left(x^{(\kappa)},\dots,x^{(1)},v,t\right) \\
  & =  \int_{\R^d} \left[ f\left(x^{(\kappa)},\dots,
    x^{(1)},w,t\right) -f\left(x^{(\kappa)},\dots,x^{(1)},v,t
    \right) \right]
    K\left(x^{(\kappa)},\dots,x^{(1)},v,w,t\right) \dd w,
\end{align*}
where the kernel $K$ satisfies pointwise bounds
\begin{multline}
  \forall \, \left(x^{(\kappa)},\dots,x^{(1)},v,w,t\right)
  \in \R^{(\kappa+2)d+1}, \\
  \label{eq:elliptic-nonloc}
  \Lambda^{-1} \big\vert v - w\big\vert^{-(d+2\beta)} \leq
    K\left(x^{(\kappa)},\dots,x^{(1)},v,w,t\right) \leq \Lambda
    \big\vert v - w\big\vert^{-(d+2\beta)},
\end{multline}
and the symmetry
\begin{multline}
  \forall \, \left(x^{(\kappa)},\dots,x^{(1)},v,w,t\right)
    \in \R^{(\kappa+2)d+1},\\
  \label{eq:symmetry-nonloc}
  K\left(x^{(\kappa)},\dots,x^{(1)},v,w,t\right) =
  K\left(x^{(\kappa)},\dots,x^{(1)},w,v,t\right).
\end{multline}

Such an operator \(\mathcal L\) is of the usual form of nonlocal diffusion
operators and can be related to Lévy processes.  Our results can be adapated to
this setting, even though the operator \(\mathcal L\) is not of the form studied
in \eqref{eq:kolmo-nonloc}.  We will not specify this direction further in this
work, since our focus is the development of an abstract method that is robust
enough to treat a wide class of operators.  The robustness of the method,
however, is best demonstrated within a class of equations that can be described
under a simple notation.

\subsection{General assumptions (H)}
We denote the variable $z:=(x,t) \in \R^{N+1}$ with $t \in \R$ and
$x=(x^{(\kappa)},\dots,x^{(1)},v)$ with each
$x^{(i)} \in \R^{d_i}$, $d_i \in \N^*$ and
$i \in \{1,\dots,\kappa\}$, and $v \in  \R^{d_0}$. We assume that
$d_0 \geq d_1 \geq \dots \geq d_{\kappa}\geq 1$ and denote
$N:=d_0+\cdots+d_\kappa$.

We consider~\eqref{eq:hormander} with operators $\cA$ and $\cB$
as follows. The operator $\cB$ is given by
\begin{equation*}
  \cB :=  (\sB x) \cdot \nabla_x,
\end{equation*}
where the $N \times N$ matrix $\sB$ has the form ($\OO$ denotes a
block matrix of zeros)
\begin{equation} \label{B_0}
  \sB
  :=
  \begin{pmatrix}
    \OO &  \sB_\kappa & \OO & \cdots & \cdots & \OO \\
    \vdots & \OO &  \sB_{\kappa-1} & \OO & \cdots & \OO\\
    \vdots & \vdots  & \OO & \ddots & \OO & \OO \\
    \vdots & \vdots  & \vdots & \ddots & \sB_2 & \OO \\
    \vdots & \vdots & \vdots & \vdots & \OO & \sB_1 \\
    \OO & \OO & \cdots & \cdots & \cdots & \OO
  \end{pmatrix} ,
\end{equation}
with $d_{i} \times d_{i-1}$ blocks $\sB_i$ which are constant
matrices of rank $d_i$ with $|\sB|\leq \Lambda$.

The operator $\cA$ is a linear operator acting only on the
variable $v$ and $\beta$-homogeneous in this variable with
$\beta \in (0,1]$ and it satisfies the following pointwise
control from above and integral control from below, given
$\varphi \in C^\infty_c(\R^N)$ such that $\| \varphi \|_{L^\infty}\geq 1$:
\begin{equation}
  \label{eq:hypA}
  \| \cA \varphi \|_\infty \lesssim \| \nabla_{v} \varphi
  \|_{L^\infty} ^{\beta+\epsilon} \| \varphi \|_{L^\infty} ^{1-\beta-\epsilon}, \qquad
  \int \left| \cA \varphi \right| \dd z \gtrsim
   \int \left| (-\Delta_v)^{\frac{\beta}{2}} \varphi \right| \dd z,
\end{equation}
where $\epsilon = 0$ for $\beta = 1$, and $\epsilon \in (0, 1-\beta)$ for $\beta \in (0, 1)$.
We point out that in case of $\beta = 1$, we identify with a slight abuse of notation $(-\Delta_{v})^{\frac{1}{2}}$ with $\nabla_{v}$, such that we are \textit{not} violating the incorrect inequality $\int \left| \nabla \varphi\right| \dd z \gtrsim  \int \left|  (-\Delta_v)^{1/2} \varphi \right| \dd z$. In fact, in the local case where $\beta = 1$, there is no need to introduce fractional derivatives at all, but in order to treat $\beta \in (0, 1]$ we do so to ease the notation.

The transport operator $\mathcal T := \partial_t + \cB$
naturally defines the non-commutative group of
\emph{hyper-Galilean transformations}, that generalises the
Galilean transformations:
\begin{equation}
  \label{grouplaw}
  (\tilde x,\tilde t) \circ (x, t) = (x + \exp(t \sB) \tilde x,
  t+ \tilde t ),
  \hspace{5mm} (x,t), \ (\tilde x, \tilde t) \in \R^{N+1}.
\end{equation}
This transformation leaves the operator $\mathcal T$ invariant:
$(\mathcal T f)(\tilde z \circ z) = \mathcal T[f(\tilde z \circ
\cdot)](z)$.
When considering the principal part operator associated to \eqref{eq:hormander}, which is obtained by replacing $A$ in \eqref{eq:kolmo-k} with the identity matrix of dimension $d$ in the local case or by choosing $a$ in \eqref{eq:kolmo-nonloc-k} as the identity in the non-local case, the related equation naturally
defines scaling properties encoded by the following \emph{hypoelliptic dilations}, defined for $r >0$
\begin{equation}
  \label{e-dilations}
  \delta_r :=\textrm{diag}\left(r^{1+2\kappa \beta}
  \mathbb{I}_{d_\kappa}, \dots,r^{1+
    2\beta}\mathbb{I}_{d_1},r \mathbb{I}_{d_0},r^{2\beta}\right),
\end{equation}
and we have the relation
\begin{equation*}
  (\cT +  \cA^* \cA)[f\circ \delta_r] = r^{2\beta}
  \big[\left(\cT + \cA^* \cA  \right)f\big] \circ \delta_r.
\end{equation*}

These transformations allow to define the \emph{hypoelliptic
  cylinders}
\begin{eqnarray*}
  \label{rcylind}
  Q_r(\tilde z):=\tilde z
  \circ\left(\delta_r\left(Q_1\right)\right) \quad \text{ with }
  \quad Q_1:=B_1 \times \ldots \times
  B_1 \times (-1,0] \subset \R^{N+1}.
\end{eqnarray*}

The (local) Kolmogorov equation~\eqref{eq:kolmo} corresponds to
$\kappa=1$, $d_0=d_1=d$,
$(x^{(1)},x^{(0)})=(x,v) \in \R^d \times \R^d$,
$\cA:=\sqrt{A} \nabla_v $ with a $d \times d$-matrix
$A=A(x,v,t)$ measurable symmetric so that
$A \in [\Lambda^{-1},\Lambda]$, and the matrix
\begin{equation}
  \label{eq:matrixB1}
  \sB := \begin{pmatrix}
    \OO & \mathbb{I}_d \\
    \OO & \OO
  \end{pmatrix}.
\end{equation}

The (local) higher-order Kolmogorov equation~\eqref{eq:kolmo-k}
corresponds to $d_i=d$ for all $i$,
$(x^{(\kappa)},\dots,x^{(1)},v) \in \R^{(\kappa+1)d}$,
$\cA:=\sqrt{A} \nabla_{v} $ with a $d \times d$-matrix
$A=A(x^{(\kappa)},\dots,x^{(1)},v,t)$ measurable symmetric
so that $A \in [\Lambda^{-1},\Lambda]$, and the matrix
\begin{equation*}
  \sB := \begin{pmatrix}
    \OO &  \mathbb{I}_d & \OO & \ldots  & \OO \\
   \OO & \OO & \mathbb{I}_d & \OO& \vdots  \\
    \OO & \vdots & \ddots  & \ddots  & \OO \\
    \vdots & \vdots  & \vdots  & \ddots & \mathbb{I}_d  \\
    \OO & \OO & \ldots & \ldots & \OO
  \end{pmatrix}.
\end{equation*}

The fractional Kolmogorov equations~\eqref{eq:kolmo-nonloc}
and~\eqref{eq:kolmo-nonloc-k} are obtained along the same line
with $\cA:= \sqrt{a} (-\Delta_v)^{\beta/2}$ and the same
matrices $\sB$.

\subsection{Main result}
\label{sec:aim}
Let us first give the definitions of weak sub-solutions and sub/super-solutions used in the results.
\begin{definition}
\label{def_weak_sol}
Let $\Omega=\Omega_x\times \Omega_t=\Omega_{(N-1)}\times \cdots \times \Omega_{(1)}\times \Omega_{(0)}\times \Omega_t$ be an open subset of $\mathbb{R}^{N+1}$.
A function $f$ is a \emph{locally integrable weak sub-solution} of \eqref{eq:hormander} if $f$ and $\cA f$ are in $L^{1}_{loc}(\Omega)$ and satisfies
\begin{align*}
\partial_t f + \cB f + \cA^* \cA f \leq 0,
\end{align*}
in the distributional sense. A function $f$ is a \emph{sub-solution} of \eqref{eq:hormander} if $f$ is a weak sub-solution in $L^\infty(\Omega_{t};L^2(\Omega_x))\cap L^2(\Omega_{(N-1)}\times \cdots \times \Omega_{(1)}\times \Omega_t; H^{\beta}(\Omega_{(0)}))$ and satisfies 
\begin{equation}\label{eq:weak-subsol}
\partial_t f + \cB f\in L^2(\Omega_{(N-1)}\times \cdots \times \Omega_{(1)}\times \Omega_t; H^{-\beta}(\Omega_{(0)})).
\end{equation}
A function $f$ is a \emph{super-solution} of \eqref{eq:hormander} if $-f$ is a \emph{sub-solution}.
\end{definition}
\begin{remark}
The condition \eqref{eq:weak-subsol} is used in \cite{GIMV} and it is weaker than the condition $$\partial_t f + \cB f\in L^{2}(\Omega),$$ used in \cite{PP}. In the definition of sub-solution, one could also consider a third condition, only valid for the local case, that reads as follows:
$$\forall G:\mathbb{R}\rightarrow \mathbb{R} \textrm{ in } C^2\textrm{ with } G'\geq 0, G''\geq 0 \textrm{ both bounded,  } G(f) \textrm{ is a weak sub-solution}.$$
This condition was introduced in \cite{guerand-mouhot} and is weaker than the other two listed in the definition above. Indeed, it allows for instance to consider $f=f(t)=\mathbf{1}_{t\leq 0}$ as a sub-solution which explains the fact that the cylinders are disjoint in the following theorems. 
\end{remark}

Our first result is the Poincaré inequality:
\begin{theorem}
  \label{thm:poincare}
  Consider operators $\cA, \cB$ satisfying the assumptions {\bf
    (H)}. Let $R >0$ be sufficiently large, depending on the
  number of commutators $\kappa$ and $\Lambda$. Let  \(Q_R :=\delta_R\left(Q_1\right)\)
  and suppose $f$ is a weak non-negative sub-solution to
  \eqref{eq:hormander} in $\tilde \Omega$ where $\tilde \Omega = Q_R$  if $\beta = 1$, that is in the local case, and  \(\tilde \Omega = B_R\times \dots \times B_R^{1 + 2(\kappa-1)\beta} \times  \R^{d_0}  \times (-5,0)\), if $\beta \in (0, 1)$, that is in the non-local case. Then
  \begin{equation}
    \label{eq:poincare}
    \int_{Q^+} \left( f - \frac{1}{|Q^-|} \int_{Q^-}
    f \right)_+  \dd z \lesssim
    \int_{\tilde \Omega} \big\vert \cA f\big\vert \dd z,
  \end{equation}
  where $Q^+ := Q_1 \subset \tilde \Omega$ denotes the future cylinder,
  and
  $Q^-=B_1 \times \ldots \times B_1 \times (-5,-4] \subset
  \tilde\Omega$ denotes the past cylinder (see
  Figure~\ref{fig:main-result}). 
  The constant is universal and
  depends on $\kappa$, $\beta$, $d_0,\dots,d_\kappa$, $\sB$, $R$
  and $\Lambda$.
\end{theorem}

\begin{figure}
  \centering

  \begin{tikzpicture}[scale =1]
    \draw [black] (-3,0) -- (3,0);
    \draw[black] (-3,-5) -- (3,-5);
    \draw [black](-3,0) -- (-3,-5);
    \draw [black](3,0) -- (3,-5) node[anchor=north, scale=1]
    {$\tilde \Omega$};
    \draw [teal](-1,0) -- (1, 0);
    \draw [teal](-1,-2) -- (1, -2);
    \draw [teal](-1,0) -- (-1, -2);
    \draw [teal](1,0) -- node[anchor=west, scale=1]
    {$Q^+$}(1, -2);
    \draw[violet] (-1,-3) -- (1, -3);
    \draw [violet](-1,-5) -- (1, -5);
    \draw [violet](-1,-3) -- (-1, -5);
    \draw [violet](1,-3)  -- node[anchor=west, scale=1]
    {$Q^-$} (1, -5);
  \end{tikzpicture}

  \caption{The different cylinders in the statement in Poincaré
    inequality (the time variable is represented vertically and
    upward).}\label{fig:main-result}
\end{figure}
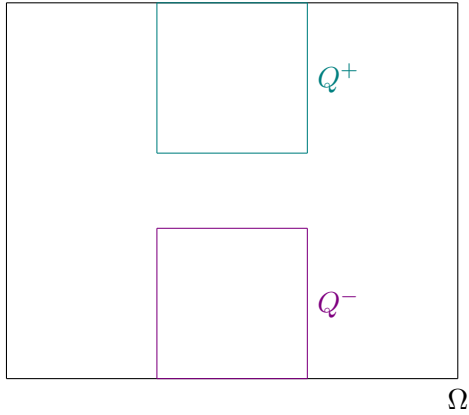

\begin{remarks}
  \begin{enumerate}
  \item The radius $R$ of the ambient space $\tilde \Omega$ is large enough
    such that the trajectories constructed in the proof (see
    \eqref{eq:control-solved}) do not exit $\tilde \Omega$. Since all maps
     involved in the construction of such trajectories are affine, we know that $R < +\infty$. We did not try to optimise the size of the ambient space. 
     
  \item It is easy to include a source term $S$ in the right hand
    side of~\eqref{eq:hormander}. Then the $L^1$ norm of $S$ on
    $\tilde \Omega$ should appear on the right hand side
    of~\eqref{eq:poincare}.

  \item It is likely that the approach can be extended to the
    case where the matrix $\sB$ has non-zero terms above the
    sub-diagonal, since they do not break the commutator
    structure. This, however, would require additional techniques,
    which we leave for further work.

  \item Under the same assumptions, it could be possible to
    prove $L^p$ versions of the Poincaré
    inequality~\eqref{eq:poincare} by the same method, up to a suitable change in the choice of the trajectory, for
    $p \in (1, + \infty)$:
    \begin{equation*}
      \int_{Q^+} \big(f - \langle
      f\rangle_{Q^-}\big)_+^p \dd z \lesssim
      \int_{\tilde \Omega} \big\vert \cA f\big\vert^p \dd z. 
    \end{equation*}
   Furthermore, it is even possible to adapt our proof to 
    obtain a gain of integrability for $f$, see the
    related comments in Remark~\ref{rmk_gain_int}.

  \item The left hand side of~\eqref{eq:poincare} could be
    mollified as follows
    \begin{equation}
      \label{eq:moll-Poincare}
      \int_{\tilde \Omega} \left(f - \int_{\tilde \Omega} f \psi_- \right) \psi_+
      \dd z \lesssim_{\psi_\pm} \int_{\tilde \Omega} \big\vert \cA
      f\big\vert\dd z
    \end{equation}
    where $\psi_\pm \in C^\infty_c(Q^\pm)$ are non-negative and
    integrate to one. This amounts to replacing respectively the
    indicator functions
    $\chi_{f(\cdot)\geq \langle f \rangle_{Q^-}} \chi_{Q^+}$ and
    $\chi_{Q^-}$ by some mollified versions $\psi_+$ and $\psi_-$
    in the proof.

  \item In such a mollified formulation, the Poincaré inequality
    in the local case would follow from the following
    \textit{``wholespace hypoelliptic Bogovoskiǐ result''} (in the spirit
    of~\cite{MR553920}): given $g=\psi_+-\psi_-$ with $\psi_\pm$
    as above, construct a vector field $\vec{F}:= (F_T,F_{v})$
    with $F_T$ one-dimensional and $F_{v}$ $d_0$-dimensional,
    so that
    \begin{equation*}
      \cT^* F_T - \nabla_{v} \cdot F_{v} = g \quad \text{ and
      } \quad F_T \ge 0 \quad \text{ and } \quad \| \nabla_{v}
      F_T \|_\infty + \| F_{v} \|_\infty \lesssim_g 1.
    \end{equation*}
    This is an open question, but we believe this vector field to
    exist.

  \item It is also possible to prove the local case of our
    Poincaré inequality through a ``\emph{parabolic wholespace
      Bogovoskiǐ inequality}'' in
    $(x^{(\kappa)},\dots,x^{(1)}, t)$ (excluding the $x^{(0)}=v$
    variable, i.e. the velocity in case of a Kolmogorov
    equation), by following the arguments
    in~\cite{dietert2023quantitative}. This parabolic wholespace
    Bogovoskiǐ inequality is an easy generalisation of the
    Bogovoskiǐ inequality, where the explicit representation of
    Bogovoskiǐ is used to show that the first coordinate of the
    vector field solving the divergence problem is controlled
    below, due to the form of $g$ given in the previous point (6).
  \end{enumerate}
\end{remarks}

It is possible to employ such a Poincaré inequality to prove a weak Harnack inequality
for weak super-solutions to equation \eqref{eq:hormander}. We only sketch the proof of this here (see Section \ref{sec:DGlemmas} for further details), since it follows from previous literature. Nevertheless, 
it is one of the main blocks to complete the De Giorgi-Nash-Moser theory for weak solutions, which 
consists in proving local quantitative regularity results (such as the a priori boundedness of weak solutions, the intermediate value lemma and
the Harnack inequality) for weak solutions, and for this reason it is now stated as an independent result.
\begin{theorem}[Weak Harnack Inequality]
  \label{thm:weak_harnack_nonlocal}
  Let the operators $\cA, \cB$ be such that $\cB$ is given by
  {\bf (H)} and $\cA$ is more specifically given by
  \begin{itemize}
  \item[(i)] $\cA:=\sqrt{A} \nabla_{v}$ with
    $A \in [\Lambda^{-1},\Lambda]$ measurable symmetric matrix (local
    case),
  \item[(ii)] $\cA = \sqrt{a} (-\Delta_v)^{\beta/2}$ with
    $a \in [\Lambda^{-1},\Lambda]$ scalar measurable (non-local
    case).
  \end{itemize}
  Let $f$ be a super-solution to~\eqref{eq:hormander} on $Q_1$, such that in the local case, $0 \leq f $ in $Q_1$, and such that in the non-local case, 
  $0 \le f \le 1$ on $B_1^{\otimes \kappa} \ \times \R^{d_0} \times [-1, 0]$. Then there is $C > 0$
  and $\zeta > 0$ depending on $\beta, N, \kappa, \Lambda$ such
  that for $r_0 > 0$ sufficiently small the Weak Harnack
  inequality is satisfied:
  \begin{equation}
    \Bigg(\int_{\tilde Q^-_{r_0/2}} f^\zeta(z) \dd
    z\Bigg)^{\frac{1}{\zeta}} \leq C \inf_{Q_{r_0/2}} f,
  \end{equation}
  where
  $\tilde Q_{\frac{r_0}{2}}^- := Q_{\frac{r_0}{2}}\big(0, \dots,
  0, -\frac{5}{2} r_0^{2\beta} +\frac{1}{2}
  \big(\frac{r_0}{2}\big)^{2\beta}\big)\in \R^{N+1}$ (the local
  case is $\beta=1$).
\end{theorem}
\begin{remark}
  \begin{enumerate}
  \item In the local case the strong Harnack follows immediately
    from the weak Harnack inequality above and the gain of
    integrability (first lemma of De Giorgi). By contrast, the
    strong Harnack inequality does not follow as immediately in
    the non-local case, since the function values outside the
    domain affect the solution inside the domain. It is however
    possible to adapt the recent work~\cite{APP, Loher-2023, Loher-2024}, which
    proves the strong Harnack inequality in the non-local case
    with $\kappa = 1$, to our setting.
  \item It is standard to deduce the Hölder regularity from the
    weak Harnack inequality. We do not repeat the argument, but it
    is identical for instance to those
    in~\cite{guerand-mouhot,loher-2022-preprint-quantitative-de-giorgi}.
    \item The boundedness of $f$ in the non-local case almost everywhere in the $x^{(0)}=v$ direction is required to make sense of the equation \eqref{eq:hormander}. 
  \end{enumerate}
\end{remark}

\subsection{Motivation}
\label{sub-motivation}

Kolmogorov equations appear in the theory of stochastic
processes: equation~\eqref{eq:kolmo} was studied by
Kolmogorov in 1934~\cite{kolmogorov}, when $A$ is the identity
matrix, to understand the ``time-integrated Brownian motion''. The
equation is then the Kolmogorov forward equation for the
process
\begin{equation*}
    {\rm d} V_t = {\rm d} W_{t}, \quad
    {\rm d} X_t =  V_{t} \dd t.
\end{equation*}
Kolmogorov obtained the explicit formula for its fundamental solution
in~\cite{kolmogorov} which inspired the seminal
work~\cite{MR0222474}. Integrating twice in time the Brownian motion
would lead to equation~\eqref{eq:kolmo-k} with $\kappa=2$ and $A$
the identity matrix. Due to this connection to stochastic
processes, several mathematical models involving linear and nonlinear Kolmogorov type equations have also appeared in finance, and in particular equations of 
type~\eqref{eq:kolmo-k} appear in various models for pricing of
path-dependent financial instruments. For
example the equation for $P=P(S,A,t)$
\begin{equation}
  \label{bpv}
  \p_t P + \tfrac12 \s^{2} S^2 \p^2_S P + (\log S ) \p_{A} P + r
  (S \p_{S} P - P) = 0, \qquad S > 0,
  \, A, t \in \mathbb{R},
\end{equation}
arises in the Black and Scholes option pricing problem
\begin{equation*}
  \begin{cases}
    {\rm d} S_t = \mu S_{t} \dd t + \s S_{t} \dd W_{t}, \\
    {\rm d} A_t =  S_{t} \dd t,
  \end{cases}
\end{equation*}
where $\s$ is the volatility of the stock price $S$, $\mu$ is the
interest rate of a riskless bond and $P= P(S, A, t)$ is the price
of the Asian option depending on the price of the stock $S$, the
geometric average $A$ of the past price and the time to maturity
$t$.  For a more exhaustive treatment of the applications of
Kolmogorov operators to finance and to stochastic theory, we
refer to the monograph \cite{Pascucci} by Pascucci.

Let us close this paragraph by mentioning that analogously,
in the non-local case when $a = 1$,
or equivalently when $\mathcal A$ corresponds to
the fractional Laplacian, the stochastic
process underlying this fractional diffusion is a Lévy process.

\subsection{Strategy of proof}

The starting point is the method developed
in~\cite{guerand-mouhot} for proving the Poincaré inequality
using trajectories and deducing the second Lemma of De Giorgi and
the weak Harnack inequality.

The strategy to obtain a Poincaré inequality for the local case
when $\kappa =1$ in~\cite{guerand-mouhot} can be summarised as
follows. The first step is to mollify the characteristic function
of $Q^-$, by adding an error term
\begin{align*}
  \int_{Q^+} \left(f-\langle f
  \rangle_{Q^{-}}\right)_+ \dd z \le
  & \int_{z_+ \in Q^+} \left( \int_{z_- \in Q^-} \left[ f(z_+)-
    f(z_-) \right] \varphi_\var(y,w) \dd z_- \right)_+ \dd z_+ \\
  & +  \var^{2d} \|f\|_{L^2(Q_R)}
\end{align*}
with $0 \leq \varphi_\var \leq 1$ smooth only depending on $(y,w)$, $\varepsilon > 0$,
$z_+=(x,v,t) \in Q^+$ and $z_-=(y,w,s) \in Q^-$. The second step
consists in constructing piecewise affine trajectories following
the two vector fields
$\mathcal T := \partial_t + v \cdot \nabla_x$ and $\nabla_v$
to connect any $z_+ \in Q^+$ with any $z_- \in Q^-$:
\begin{align*}
  (x,v,t)  \underset{\nabla_v}{\longrightarrow}
  \left(x,\frac{x-y}{t-s},t\right)
  \underset{\cT}{\longrightarrow} \left(y,\frac{x
  -y}{t-s},s\right) \underset{\nabla_v}{\longrightarrow}
  (y,w,s).
\end{align*}
Note here that a positive time gap between the two cylinders
$Q^-$ and $Q^+$ is used to make sure the intermediate velocity
$(x-y)/(t-s)$ remains bounded. We then write $f(z_+) - f(z_-)$ as
an integral along the chosen trajectory and use
$\cT f + \cA^* \cA f \le 0$:
\begin{align*}
  \int_{Q^-} \left[ f(z_+) - f(z_-)\right]
  \varphi_\var(y,w) \dd z_-
  & = \int_{Q^-} \int_{\text{trajectory}}
    \cT f(\cdots) \varphi_\var(y,w) \dd s \dd z_- \\
  & \quad + \int_{Q^-} \int_{\text{trajectory}} \nabla_v f(\cdots)
    \varphi_\var(y,w) \dd s \dd z_- \\
  & \lesssim \int_{Q^-} \int_{\text{trajectory}}
    \left[ \nabla_v \cdot \left( A \nabla_v f \right) \right](\cdots)
    \varphi_\var(y,w) \dd s \dd z_- \\
  & \quad + \int_{Q^-} \int_{\text{trajectory}} \nabla_v f(\cdots)
    \varphi_\var(y,w) \dd s \dd z_-.
\end{align*}
We then want to integrate by parts the $v$-divergence in the
past variable $z_-$. This integration by parts degenerates near
the future point $z_+$, since $z_+$ does not depend on $z_-$, and
it produces a non-integrable singularity. In order to overcome
this singularity, the paper~\cite{guerand-mouhot} introduced an
additional small fourth sub-trajectory along $\nabla_x$:
\begin{align*}
(x,v,t)
  & \underset{\nabla_x}{\longrightarrow}
    (x+\varepsilon w,v,t) \underset{\nabla_v}{\longrightarrow}
    \left(x+\varepsilon w, \frac{x+\varepsilon w
    -y}{t-s},t \right) \\
  & \underset{\cT}{\longrightarrow} \left(y,\frac{x+\varepsilon w
    -y}{t-s},s\right) \underset{\nabla_v}{\longrightarrow}
    (y,w,s).
\end{align*}

This sub-trajectory ``noises'' the future position variable with
the past velocity. This produces a second error term of the form
$\var^\sigma \| f \|_{L^1_{t,v} W^{\sigma,1}_x}$, which can, however,
 be controlled for a small $\var$ by the
integral regularity $L^1_{t,v} W^{\sigma,1}_{x}$ (which can be
established for sub-solutions to the Kolmogorov equation, for a
small $\sigma>0$). Finally it
yields
\begin{align*}
  \int_{Q^+} &\big(f - \langle
  f\rangle_{Q^-}\big)_+  \dd z\\
  &\qquad \lesssim
  C_\var \int_{Q_R} \big\vert S\big\vert \dd z +
  \int_{Q_R} \big\vert \nabla_v f\big\vert \dd z + \var^\sigma \| f
  \|_{L^1_{t,v} W^{\sigma,1}_x(Q_R)} + \var^{2d} \|f\|_{L^2(Q_R)}
\end{align*}
for a constant $C_\var >0$ depending on $\var>0$. The two
additional errors weaken the Poincaré inequality, but the latter
is shown to be sufficient for implementing the next steps of the
De Giorgi theory in~\cite{guerand-mouhot}.

In the present article, we improve the method by simplifying the
proof and removing the two error terms in the inequality.
Moreover, the gain of integrability can directly be extracted
from the proof of the Poincaré inequality.
The first novel idea is to use three cylinders $Q^-$, $Q^0$ and
$Q^+$, and mollify the characteristic function only in the
intermediate cylinder $Q^0$, see Figure~\ref{fig:trajectories}.
\begin{figure}{
\centering
\begin{tikzpicture}[scale =0.8]
  \draw [black] (-4,0) -- (4,0);
   \draw[black] (-4,-12) -- (4,-12);
  \draw [black](-4,0) -- (-4,-12);
  \draw [black](4,0) -- (4,-12) node[anchor=north, scale=1]
  {\footnotesize{$\tilde \Omega$}};
   \draw [teal](-1,0) -- (1, 0);
   \draw [teal](-1,-2) -- (1, -2);
   \draw [teal](-1,0) -- node[anchor=east, scale=1] {$Q^+$} (-1,
   -2);
   \draw [teal](1,0) -- (1, -2) ;
    \draw[olive] (-1,-5) -- (1, -5);
   \draw [olive](-1,-7) -- (1, -7);
   \draw [olive](-1,-5) -- (-1, -7);
   \draw [olive](1,-5) -- node[anchor=west, scale=1] {$Q^0$}(1,
   -7);

    \draw[violet] (-1,-10) -- (1, -10);
   \draw [violet](-1,-12) -- (1, -12);
   \draw [violet](-1,-10) -- (-1, -12);
   \draw [violet](1,-10)  -- node[anchor=west, scale=1] {$Q^-$}
   (1, -12);
   \draw [->, cyan, dashed, line width=0.5mm]  plot [smooth,
   tension = 0.7] coordinates {(0,-1)(3.6, -1.7)  (0,-6)};
   \node [cyan] at (2.8, -2.5) {$\Gamma_+$};
   \draw [->, magenta, dashed, line width=0.5mm] plot [smooth,
   tension = 0.7] coordinates {(0,-11)(-3.6, -10.3) (0,-6)};
    \node [magenta] at (-2.8, -8.3) {$\Gamma_-$};
   \coordinate[label={[teal]above:\footnotesize{$z_+$}}] (z_+) at
   (0, -1);
   \draw [teal] plot[only marks,mark=x,mark size=2pt]
   coordinates {(0, -1)};
   \coordinate[label={[olive]above:\footnotesize{$z_0$}}] (z_0)
   at (0, -6);
   \draw [olive] plot[only marks,mark=x,mark size=2pt]
   coordinates {(0, -6)};
   \coordinate[label={[violet]above:\footnotesize{$z_-$}}] (z_-)
   at (0, -11);
   \draw [violet] plot[only marks,mark=x,mark size=2pt]
   coordinates {(0, -11)};
\end{tikzpicture}
}
\caption{Construction of the trajectories. The curve $\Gamma_+$
  connects any point $z_+ \in Q^+$ to some intermediate point
  $z_0 \in Q^0$, whereas the curve $\Gamma_-$ connects any point
  $z_- \in Q^-$ to some intermediate point
  $z_0 \in Q^0$.}\label{fig:trajectories}
\end{figure}
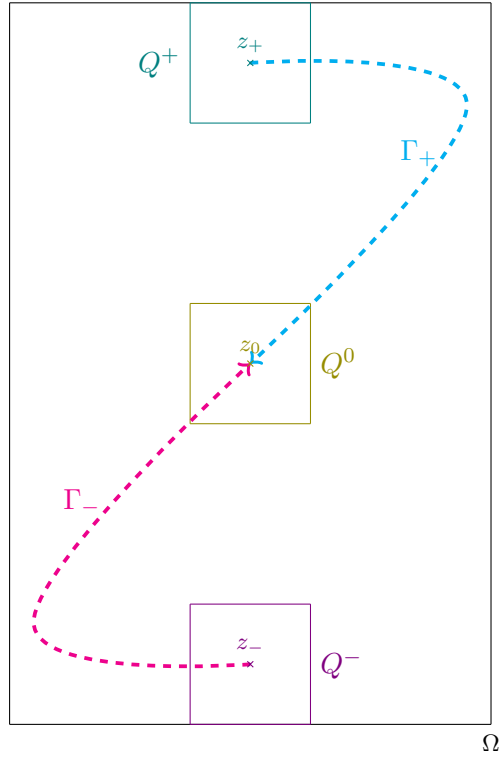
It is easy to deduce from the triangle inequality
\begin{equation}\label{eq:split-Q0}
  \begin{aligned}
    \fint_{Q^+} \left[ f(z_+) - \langle f \rangle_{Q^-} \right]_+ \dd z_+
    &\le
    \fint_{Q^+} \left( \fint_{Q^0} \left[ f(z_+) -
        f(z_0) \right] \varphi_\var(x_0, v_0)\dd z_0 \right)_+ \dd z_+ \\
    &\quad + \fint_{Q^-} \left( \fint_{Q^0} \left[ f(z_0) -
        f(z_-) \right]\varphi_\var(x_0, v_0) \dd z_0 \right)_+ \dd z_-
  \end{aligned}
\end{equation}
for $\varphi_\var \ge 0$ smooth, with mass $1$, where
$z_0=(x_0,v_0,t_0) \in Q^0, z_+ \in Q^+, z_- \in Q^-$, and where
we have denoted by $\fint_{Q} := |Q|^{-1} \int_Q$ the normalized
integral. We then construct trajectories to connect $z_+$ to
$z_0$, and $z_-$ to $z_0$. This removes the first error term we
had introduced in relation to the mollification $\varphi_\var$.

The second novel idea is to connect the points by a curved
trajectory, rather than a piecewise affine one, by solving a
control problem with a well-chosen forcing (i.e. control
function), in order to reduce the order of the singularity of the
integration by parts near $z_{\pm}$. More precisely, the first
idea is to consider trajectories whose speed diverges at
$z_{\pm}$ but with an integrable divergence so that the
trajectories remain bounded. This idea is inspired from the
reading of~\cite{niebel2022kinetic}, but we propose simpler and
more systematic trajectories. The second important idea is to use
$\kappa+1$ linearly independent control power functions with such
well-balanced diverging behaviour at $z_{\pm}$ in order to
connect all points between the cylinders, thanks to the Hörmander
commutator condition.

\subsection{Outline}

In \ref{sec:trajectories} we provide the explicit construction
of the trajectories connecting points in the future to points in
the past. These trajectories are then used in
Section~\ref{sec:proof} to prove the Poincaré inequality; they
allow to estimate the $L^1$ norm in the future of the difference
between $f$ and the past average of $f$, or in other words, the
left hand side of \eqref{eq:poincare}. In \ref{sec:DGlemmas}, we
prove the Harnack inequalities on the basis of this Poincaré
inequality.  
\\
\\
\paragraph{\textbf{Acknowledgements}}
We are grateful to Lukas Niebel and Rico Zacher for stimulating discussions and inspiring ideas on the subject of this article.

\section{Construction of the trajectories}
\label{sec:trajectories}

\subsection{The base case}
\label{sec:toy-local}

Let us first consider the case with one commutator ($\kappa=1$),
$d_1=d_0=d$ and $\sB_1=\I$. Given three points
$z_+ \in Q^+, z_0 \in Q^0$ and $z_- \in Q^-$, we want to
construct two paths $s \rightarrow \Gamma_+(s)$ and
$s \rightarrow \Gamma_-(s)$ for $s \in [0, 1]$ such that (see
Figures~\ref{fig:trajectories} and~\ref{fig:trajectories-toy})
\begin{equation*}
  \begin{aligned}
    & \Gamma_+(s) = \left(X^{(1)}_+(s), X^{(0)}_+(s),
    T_+(s)\right)^T, \qquad \Gamma_+(0) = z_+, \qquad \Gamma_+(1) =
    z_0,\\
    & \Gamma_-(s) = \left(X^{(1)}_-(s), X^{(0)}_-(s),
    T_-(s)\right)^T, \qquad \Gamma_-(0) = z_-, \qquad \Gamma_-(1) =
    z_0,
  \end{aligned}
\end{equation*}
where
\begin{equation}
  \label{eq:control}
  \begin{aligned}
    \begin{cases}
      \ds X^{(0)}_{\pm}(s) 
      = \mathsf m_\pm ^{(0)}
      g_0''(s) +\mathsf m^{(1)}_\pm
      g_1 ''(s), \\[2mm]
      \ds X^{(1)}_\pm(s) = \delta_\pm X^{(0)}_\pm(s),\\[2mm]
      \ds T_\pm(s) = \delta_\pm,
    \end{cases}
  \end{aligned}
\end{equation}
where $\delta_\pm := t_0 - t_\pm$ and the control functions
$g_i \in C^2((0,1])$ with $g_i(0)=g_i'(0)=0$ for $i=0,1$, and
$\mathsf m^{(i)} _\pm \in \R^d$ for $i=0,1$.
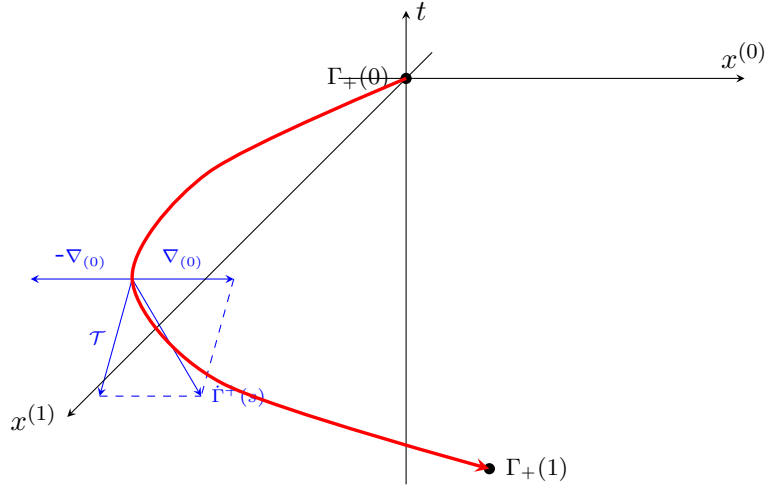
\begin{figure}{ \centering
    \begin{tikzpicture}[scale = 0.45,>=stealth]
      \draw[->] (xyz cs:x=-2) -- (xyz cs:x=10) node[above] {$x^{(0)}$};
      \draw[->] (xyz cs:y=-12) -- (xyz cs:y=2) node[right] {$t$};
      \draw[->] (xyz cs:z=-2) -- (xyz cs:z=26) node[left] {$x^{(1)}$};
      \node[fill, circle,inner
      sep=1.5pt,label={left:\footnotesize{$\Gamma_+(0)$}}] at
      (xyz cs:x=0,z=0, y = 0) {};
      \node[fill,circle,inner
      sep=1.5pt,label={right:\footnotesize{$\Gamma_+(1)$}}] at
      (xyz cs:x=4,z=4, y = -10) {};
      \draw[smooth, ->, red, line width=0.45mm] plot coordinates
      {(xyz cs:x=0,z=0, y = 0) (xyz cs:x=-5.56,z=0.732, y = -2.5)
        (xyz cs:x=-7.16,z=2.42, y = -5)  (xyz cs:x=-3.94,z=3.91,
        y = -7.5)  (xyz cs:x=4,z=4, y = -10)};
      \begin{scope}[on background layer]
         \draw[smooth, ->, blue] (xyz cs:x=-7.16,z=2.42, y = -5)
         -- node[above] {-\tiny{$\nabla_{(0)}$}}(xyz cs:x=-10.16,
         z=2.42, y = -5);
      \draw[smooth, ->, blue] (xyz cs:x=-7.16,z=2.42, y = -5) --
      node[above] {\tiny{$\nabla_{(0)}$}}(xyz cs:x=-4.16, z=2.42,
      y = -5);
      \draw[smooth, ->, blue] (xyz cs:x=-7.16,z=2.42, y = -5) --
      node[left] {\tiny{$\cT$}}(xyz cs:x=-7.16, z=4.92, y =
      -7.5);
      \draw[smooth, ->, blue] (xyz cs:x=-7.16,z=2.42, y = -5) --
      (xyz cs:x=-4.16, z=4.92, y = -7.5)node[right] {\tiny{$\dot
          {\Gamma}^+(s)$}};
      \draw[dashed, blue](xyz cs:x=-7.16, z=4.92, y = -7.5) --
      (xyz cs:x=-4.16, z=4.92, y = -7.5);
        \draw[dashed, blue](xyz cs:x=-4.16, z=2.42, y = -5)--
        (xyz cs:x=-4.16, z=4.92, y = -7.5);
    \end{scope}
    \end{tikzpicture}
  }
  \caption{Construction of the trajectories. The curve $\Gamma_+$
    connects any point $\Gamma_+(0) = z_+ \in Q^+$ to some
    intermediate point $\Gamma_+(1) = z_0 \in Q^0$ along the
    vector fields $\cT$ and
    $\nabla_{(0)}$.} \label{fig:trajectories-toy}
\end{figure}
Solving the differential equations yields
\begin{equation}
  \label{eq:control-solved}
  \begin{aligned}
    \begin{cases}
      X^{(0)}_{\pm}(s) = x^{(0)} _\pm + \mathsf m_\pm ^{(0)}
      g_0'(s) + \mathsf m^{(1)}_\pm
      g_1 '(s), \\[2mm]
      X^{(1)}_\pm(s) = x^{(1)} _\pm + s \delta_\pm x^{(0)} _\pm +
      \delta_\pm \left[ \mathsf m_\pm ^{(0)} g_0(s) + \mathsf
        m^{(1)}_\pm
        g_1 (s) \right] ,\\[2mm]
      T_\pm(s) = s t_0 + (1-s) t_\pm.
    \end{cases}
  \end{aligned}
\end{equation}

Let us denote $\U_d$ the $d$-vector of $1$'s, $\V_d$ the
$d$-vector of $0$'s, $\I_d$ the $d\times d$ identity matrix,
$\OO_d$ the $d\times d$ zero block, and
\begin{align*}
  & \mathsf M_\pm :=
  \begin{pmatrix}
    \mathsf m^{(0)}_\pm \\
    \mathsf m^{(1)} _\pm
  \end{pmatrix}, \quad
  \sW(s) :=
  \begin{pmatrix}
    g_0'(s) \I_d & g_1 '(s) \I_d \\
    g_0(s) \I_d & g_1(s) \I_d
  \end{pmatrix}, \\
  & \sW_\pm ^\delta (s) :=
  \begin{pmatrix}
    \I_d & \OO_d \\
    \OO_d & \delta_\pm \I_d
  \end{pmatrix} \sW(s), \quad
  \mathsf Y_\pm =
  \begin{pmatrix}
    \mathsf y^{(0)} _\pm \\
    \mathsf y^{(1)} _\pm
  \end{pmatrix}
  :=
  \begin{pmatrix}
    x^{(0)}_0 - x^{(0)}_\pm \\
    x^{(1)}_0 - x^{(1)}_\pm - \delta_\pm x^{(0)}_\pm
  \end{pmatrix}.
\end{align*}
The boundary conditions $\Gamma_\pm(1)=z_0$ impose
\begin{align}\label{eq:boundary-condition}
  \sW_\pm ^\delta(1) \mathsf M_\pm = \mathsf Y_\pm \quad \iff
  \quad \mathsf M_\pm = \sW_\pm ^\delta(1)^{-1} \mathsf Y_\pm,
\end{align}
provided that the Wronskian matrix $\sW(1)$ is invertible at
$s=1$. We deduce
\begin{align*}
  \begin{pmatrix}
    X^{(0)}_\pm (s) \\
    X^{(1)}_\pm (s)
  \end{pmatrix}
  =   \sW_\pm ^\delta(s) \mathsf M_\pm +
  \begin{pmatrix}
  	\I_d & \OO_d \\
	\delta_\pm s \I_d & \I_d
  \end{pmatrix}
  \begin{pmatrix}
  x_\pm^{(0)}\\
  x_\pm^{(1)}
  \end{pmatrix}.
\end{align*}
Using~\eqref{eq:boundary-condition} we obtain
\begin{align*}
  \begin{pmatrix}
    X^{(0)}_\pm (s) \\
    X^{(1)}_\pm (s)
  \end{pmatrix}
  &= \sW_\pm ^\delta(s) \sW_\pm ^\delta(1)^{-1} \mathsf Y_\pm+
        \begin{pmatrix}
          \I_d & \OO_d \\
          \delta_\pm s \I_d & \I_d
        \end{pmatrix}
        \begin{pmatrix}
          x_\pm^{(0)}\\
          x_\pm^{(1)}
        \end{pmatrix}\\
  &= W^\delta_\pm(s) \big[W^\delta_\pm(1)\big]^{-1}
    \begin{pmatrix}
      x_0^{(0)}\\
      x_0^{(1)}
    \end{pmatrix}\\
  &\quad- W^\delta_\pm(s) \big[W^\delta_\pm(1)\big]^{-1}
    \begin{pmatrix}
    x^{(0)}_\pm \\ x^{(1)}_\pm + \delta_\pm x^{(0)}_\pm
  \end{pmatrix}
  +
  \begin{pmatrix}
    \I_d & \OO_d \\
    \delta_\pm s \I_d & \I_d
  \end{pmatrix}
  \begin{pmatrix}
    x_\pm^{(0)}\\
    x_\pm^{(1)}
  \end{pmatrix} \\
  & = \mathfrak A^s \begin{pmatrix} x_0^{(0)} \\
    x_0^{(1)} \end{pmatrix} + \mathfrak B^s = \mathfrak A^s
  x_0 + \mathfrak B^s =: \Phi^s_\pm (x_0)
\end{align*}
is an affine function with matrix
$\mathfrak A^s:=\sW_\pm ^\delta (s) [\sW_\pm ^\delta (1)]^{-1}$
and a vector $\mathfrak B^s$ that depend only on $s$ and $x_\pm$.
We still have to prove that the matrix $\mathfrak A^s$ is
invertible for $s \in (0,1]$. If so, given $s \in (0,1]$, the
derivative along the first variable $X^{(0)}_\pm(s)$ of the
inverse is
\begin{equation*}
  \nabla_{(0)} \left( \Phi_\pm ^s \right)^{-1} =
  \sW^\delta_\pm(1)\big[\sW^\delta_\pm(s)\big]^{-1}
  \left( \begin{matrix}
      \U_d \\ \V_d
  \end{matrix} \right).
\end{equation*}
We choose
$g_i(s) := (1+\alpha_i)^{-1} (2+\alpha_i)^{-1} s^{2+\alpha_i}$,
$i=0,1$, with $\alpha_0, \alpha_1 \in (-1,0)$ and
$\alpha_0 \not = \alpha_1$. Then the Wronskian matrix is
invertible for all $s \not =0$:
\begin{equation}
\begin{aligned}
  \label{toy-jacobian}
  \det \sW(s)
  &= \left(g_0'(s) g_1(s)- g_0(s) g_1'(s)\right)^d \\
  &= \frac{(\alpha_0 - \alpha_1)^d}{(1+\alpha_0)^d(1+\alpha_1)^d
    (2+\alpha_0)^d(2+\alpha_1)^d} s^{(3+\alpha_0+\alpha_1)d} \not = 0.
\end{aligned}
\end{equation}
Our choice of two linearly independent control functions $g_0$
and $g_1$ is guided by ensuring the Wronskian matrix to be
invertible. Moreover,
\begin{equation*}
\begin{aligned}
   &\sW_\pm ^\delta (1) =
  \begin{pmatrix}
    \frac{\I_d}{(1+\alpha_0)} & \frac{\I_d}{(1+\alpha_1)} \\
    \frac{\delta_\pm \I_d}{(1+\alpha_0)(2+\alpha_0)} &
    \frac{\delta_\pm \I_d}{(1+\alpha_1)(2+\alpha_1)}
  \end{pmatrix}, \\
  &\sW_\pm ^\delta (s)^{-1} = \frac{1}{\delta_\pm\det \sW(s)}
  \begin{pmatrix}
    \frac{\delta_\pm s^{2+\alpha_1} \I_d}{(1+\alpha_1)(2+\alpha_1)} &
    \frac{-s^{1+\alpha_1} \I_d}{(1+\alpha_1)} \\
    \frac{-\delta_\pm s^{(2+\alpha_0)} \I_d}{(1+\alpha_0)(2+\alpha_0)}
    & \frac{s^{1+\alpha_0} \I_d}{(1+\alpha_0)}
  \end{pmatrix},
 \end{aligned}
\end{equation*}
which finally yields
\begin{align*}
  \nabla_{(0)} \left( \Phi_\pm ^s \right)^{-1}
  & = W^\delta_\pm(1)\big[W^\delta_\pm(s)\big]^{-1}
    \begin{pmatrix}
      \U_d \\ \V_d
    \end{pmatrix} \\
  & = \frac{1}{(\alpha_0-\alpha_1)}
    \begin{pmatrix}
      (2+\alpha_0) s^{-1-\alpha_0} - (2+\alpha_1) s^{-1-\alpha_1}
      \U_d
      \\
      \delta_\pm \left( s^{-1-\alpha_0} - s^{-1-\alpha_1} \right)
      \U_d
    \end{pmatrix}\\
  &= O(s^{-1-\alpha_0}) + O(s^{-1-\alpha_1}),
\end{align*}
which remains integrable for $\alpha_0,\alpha_1 \in
(-1,0)$ with $\alpha_0 \neq \alpha_1$. Observe also that our choice of control functions $g_0$
and $g_1$ implies that all their derivatives up to order two are
integrable on $s \in [0,1]$, which implies that the trajectories
are bounded (with integrable tangent vector field).

\subsection{The general case}

Let us now consider the general case with $\kappa \ge 2$
commutators, general dimensions
$d_0 \geq d_1 \geq \dots \geq d_{\kappa}\geq 1$ and surjective
$d_i \times d_{i-1}$ matrices $\sB_i$. Consider for $s \in [0,1]$
the paths
\begin{equation*}
  \begin{aligned}
    & \Gamma_+(s) = \left(X^{(\kappa)}_+(s), \dots, X^{(0)}_+(s),
    T_+(s)\right)^T, \qquad \Gamma_+(0) = z_+, \qquad \Gamma_+(1) =
    z_0,\\
    & \Gamma_-(s) = \left(X^{(\kappa)}_-(s), \dots, X^{(0)}_-(s),
    T_-(s)\right)^T, \qquad \Gamma_-(0) = z_-, \qquad \Gamma_-(1) =
    z_0,
  \end{aligned}
\end{equation*}
and the following control problem
\begin{equation}
  \label{eq:control-general}
  \begin{aligned}
    \begin{cases}
      \ds X^{(0)}_{\pm}(s) 
      = \sum_{i=0} ^\kappa
      \mathsf m_\pm ^{(i)} g_i^{(\kappa+1)} (s), \\[2mm]
      \ds X^{(1)}_\pm(s) = \delta_\pm \sB_1 X^{(0)}_\pm(s),\\[2mm]
      \quad \vdots \\[2mm]
      \ds X^{(\kappa)}_\pm(s) = \delta_\pm  \sB_\kappa X^{(\kappa-1)}_\pm(s), \\[2mm]
      \ds T_\pm(s) = \delta_\pm
    \end{cases}
  \end{aligned}
\end{equation}
with $\delta_\pm := t_0 - t_\pm$ and the control functions
$g_i \in C^{\kappa +1}((0,1])$ so that
$g_i(0)=g_i'(0)= \dots =g_i^{(\kappa)}(0)=0$ for
$i=0,1,\dots,\kappa$, and $\mathsf m^{(i)} _\pm \in \R^{d_0}$ for
$i=0,\dots,\kappa$. Then
\begin{equation}
  \label{eq:control-solved-general}
  \begin{aligned}
    \begin{cases}
      X^{(0)}_{\pm}(s) = x^{(0)} _\pm + \sum_{i=0} ^\kappa
      \mathsf m_\pm ^{(i)}
      g_i^{(\kappa)}(s), \\[2mm]
      X^{(1)}_\pm(s) = x^{(1)} _\pm + s \delta_\pm \tilde \sB_{1,1} x^{(0)} _\pm +
      \delta_\pm \left[ \sum_{i=0} ^\kappa \tilde \sB_{1,1} \mathsf
          m_\pm ^{(i)} g_i^{(\kappa-1)} (s) \right] ,\\[2mm]
      X^{(2)}_\pm(s) = x^{(2)} _\pm + s \delta_\pm \tilde
      \sB_{2,2} x^{(1)} _\pm +
      \frac{(s\delta_\pm)^2}{2} \tilde \sB_{2,1} x^{(0)} _\pm +
      \delta_\pm ^2 \left[ \sum_{i=0} ^\kappa \tilde \sB_{2,1} \mathsf m_\pm
          ^{(i)} g_i^{(\kappa-2)}(s) \right] ,\\[2mm]
      \quad \vdots \\[2mm]
      X^{(\kappa)}_\pm(s) = x_\pm ^{(\kappa)} + \sum_{i=1} ^\kappa \frac{(s
        \delta_\pm)^i}{i!} \tilde \sB_{\kappa,\kappa-i+1} x^{(\kappa-i)} _\pm + \delta_\pm
      ^\kappa \left[ \sum_{i=0} ^\kappa \tilde \sB_{\kappa,1} \mathsf
          m_\pm ^{(i)} g_i (s) \right] ,\\[2mm]
      T_\pm(s) = s t_0 + (1-s) t_\pm
    \end{cases}
  \end{aligned}
\end{equation}
where $\tilde \sB_{i,j} := \sB_i \sB_{i-1} \cdots \sB_j$ for
$1 \le j \le i$, which is a $d_i \times d_{j-1}$ block matrix.

We denote by $\OO$ a zero block matrix with arbitrary size and
\begin{align}\label{eq:matrices}
  &  \mathsf M_\pm :=
  \begin{pmatrix}
    \mathsf m^{(0)}_\pm \\
    \vdots \\
    \mathsf m^{(\kappa)} _\pm
  \end{pmatrix}, \quad
  \sW(s) :=
  \begin{pmatrix}
    g_0^{(\kappa)}(s) \I_{d_0} & \dots & g_\kappa ^{(\kappa)} (s)
    \I_{d_0} \\
    \vdots & \vdots & \vdots \\
    g_0(s) \I_{d_0} & \dots & g_\kappa(s) \I_{d_0}
  \end{pmatrix}, \\
  & \sW_\pm ^\delta(s) :=
  \begin{pmatrix}
    \I_{d_0} & \OO & \cdots & \cdots & \OO \\
    \OO & \tilde \sB_{1,1} \delta_{\pm} & \OO & \cdots & \OO \\
    \vdots & \OO &  \tilde\sB_{2,1} \delta_\pm ^2 & \ddots & \OO \\
    \vdots & \vdots & \ddots & \ddots & \OO \\
    \OO & \cdots & \cdots & \OO & \tilde \sB_{\kappa,1} \delta_\pm ^\kappa
  \end{pmatrix}
  W(s) =:  \mathsf R \,  \mathsf W(s), \\
  & \mathsf Y_\pm =
  \begin{pmatrix}
    \mathsf y^{(0)} _\pm \\
    \vdots \\
    \vdots \\
   \mathsf  y^{(\kappa)} _\pm
  \end{pmatrix}
  := \begin{pmatrix}
    \mathsf x_0 ^{(0)} \\
    \vdots \\
    \vdots \\
   \mathsf  x_0 ^{(\kappa)}
  \end{pmatrix} - \mathsf T_\pm(1) \begin{pmatrix}
    \mathsf x^{(0)} _\pm \\
    \vdots \\
    \vdots \\
   \mathsf  x^{(\kappa)} _\pm
 \end{pmatrix}, \\
  & \sT_\pm (s) := \begin{pmatrix}
    \I_{d_0} & \OO & \cdots & \cdots & \OO \\
    (s\delta_\pm) \tilde \sB_{1,1} & \Id_{d_1} & \OO & \cdots & \OO \\
    \vdots & (s \delta_\pm) \tilde \sB_{2,2} & \Id_{d_2} & \ddots & \OO \\
    \vdots & \vdots & \ddots & \ddots & \OO \\
    \frac{(s\delta_\pm)^\kappa}{\kappa!} \tilde \sB_{\kappa,1} &
    \cdots & \cdots & (s \delta_\pm) \tilde \sB_{\kappa,\kappa} & \Id_{d_\kappa}
  \end{pmatrix}.
\end{align}
Note that the matrix $\sW(s)$ is the Wronskian of the family of
functions $(g_i)_{i=0} ^\kappa$ and is invertible when they are
the linearly independent solutions to a $(\kappa+1)$-order linear
ODE. The boundary conditions $\Gamma_\pm(1)=z_0$ implies
$
  \sW_\pm ^\delta (1) \mathsf M_\pm = \mathsf Y_\pm $ which has at least the following solution
  \begin{equation*}
  	\mathsf M_\pm := \sW_\pm ^\delta (1)^{-1}\mathsf Y_\pm
  \end{equation*}
  thanks to Moore-Penrose theory, see for example \cite{Barata-Hussein}, where $(\sW_\pm ^\delta (1))^{-1}$ is the right pseudo-inverse of $\sW_\pm ^\delta (1)$.
Indeed, the right pseudo-inverse of the matrix $\sW_\pm ^\delta (1)$ always exists, and is given by the product of the inverse of $\sW(1)$ times the right pseudo-inverse of $\mathsf R$ that we denote here by ${\mathsf R}^{-1}$.
Thus we deduce that
\begin{align}
  \nonumber
  \begin{pmatrix}
    X^{(0)}_\pm (s) \\\
    \vdots \\
    X^{(\kappa)}_\pm (s)
  \end{pmatrix}
  &=  \sW_\pm ^\delta (s) \big[\sW_\pm ^\delta
    (1)\big]^{-1} x_0 
     + \left( \sT_{\pm} (s) - \sW_\pm ^\delta (s) \big[\sW_\pm ^\delta
    (1)\big]^{-1} \sT_{\pm} (1) \right)  x_\pm \\ 
  \label{eq:mapping}
  & = \mathfrak A^s x_0 + \mathfrak B^s =: \Phi^s_\pm (x_0)
\end{align}
is an affine function with matrix
$\mathfrak A^s:=\sW_\pm ^\delta (s) [\sW_\pm ^\delta (1)]^{-1} =  \mathsf R \mathsf W (s) \mathsf W(1)^{-1} \mathsf R^{-1}$
and vector $\mathfrak B^s$ that depend only on $s$ and $x_\pm$.
We still have to prove that the matrix $\mathfrak A^s$ is
invertible for $s \in (0,1]$. If so, given $s \in (0,1]$, the
derivative along the first variable $X^{(0)}_\pm(s)$ of the
inverse is
\begin{equation}
  \label{eq:first-der-k}
  \nabla_{(0)} \left( \Phi_\pm ^s \right)^{-1} = \sW_\pm ^\delta
  (1) \big[\sW_\pm ^\delta(s)\big]^{-1}
  \begin{pmatrix}
     \U_{d_0} \\ \V_{d_1} \\ \vdots \\ \V_{d_\kappa}
  \end{pmatrix}.
\end{equation}
We finally choose the control functions
\begin{equation}\label{eq:choice-g}
  g_i(s) := \frac{s^{1+\kappa+\alpha_i}}{(1+\alpha_i)(2+\alpha_i)
    \cdots (1+\kappa+\alpha_i)}
\end{equation}
with $\alpha_i \in (-1,0)$ pairwise distinct for
$i=0,\dots,\kappa$. 
Thus, we are left to prove that $\mathsf W(1)$ and $\mathfrak{A}_s$ are invertible for $s \in (0,1]$, but this boils down to prove the matrix $\mathsf W(s)$ is invertible for $s \in (0,1]$. Indeed, when $s \ne 0$ the Wronskian matrix is
invertible with the precise $s$-behaviour given by the following
lemma:

\begin{lemma}
  \label{lem:wronskian}
  Consider the $(\kappa+1)\times(\kappa+1)$ matrix
  $P(s):= (s^{1+\alpha_j+i}/p_{i,j})_{i,j=0} ^\kappa$ with
  $p_{i,j}:=(1+\alpha_j) \cdots (1+i+\alpha_j)$ and the matrix\footnote{The notation $P\otimes \I_{d_0}$ is the classical tensor product between the matrices $P$
  and $\I_{d_0}$.} $W:=P\otimes \I_{d_0}$. Then their
  determinants are given by
  \begin{equation}
    \label{eq:det-k}
    \begin{aligned}
    &\det P(s)=
    \left( \frac{ \prod_{i, j=0 | i<j}^\kappa
      (\alpha_i - \alpha_j)}{\prod_{i,j=0} ^\kappa(1+i+\alpha_j)}
 \right) s^{ \frac{(\kappa+1)(\kappa+2)}{2}+\sum_{i=0}^\kappa \alpha_i}\\
    &\det \sW(s) = \left( \frac{ \prod_{i, j=0 | i<j}^\kappa
      (\alpha_i - \alpha_j)^{d_0}}{\prod_{i,j=0} ^\kappa(1+i+\alpha_j)^{d_0}}
 \right) s^{d_0 \frac{(\kappa+1)(\kappa+2)}{2}+d_0\sum_{i=0}^\kappa \alpha_i}.
  \end{aligned}
  \end{equation}
\end{lemma}
\begin{proof}[Proof of Lemma~\ref{lem:wronskian}]
The factorisation in $s$ gives the exponents. It remains to compute the determinant of matrix ${P}(1):= (1/p_{i,j})_{i,j=0} ^\kappa$. From now on,
 for every $n\geq1$, we define $Q_n(X)=\prod_{j=0}^{n-1} (X+\kappa+1-j)$ a polynomial of degree $n$. Then, we observe:
$$\prod_{i,j=0} ^\kappa (1+i+\alpha_j) \det {P}(1)=
	\det
	\begin{pmatrix}
    Q_{\kappa}(\alpha_0) &  Q_{\kappa}(\alpha_1) & \cdots &  Q_{\kappa}(\alpha_{\kappa}) \\
      Q_{\kappa-1}(\alpha_0) &  Q_{\kappa-1}(\alpha_1) & \cdots &  Q_{\kappa-1}(\alpha_{\kappa}) \\
    \vdots & \ddots & \ddots & \vdots  \\
 Q_{1}(\alpha_0) &  Q_{1}(\alpha_1) & \cdots & Q_{1}(\alpha_{\kappa}) \\
    1 & 1 & \cdots & 1
\end{pmatrix}, $$
where the equivalence follows by multiplying the $j^{th}$-column by $\prod \limits_{i=0}^\kappa (1+i+\alpha_j)$ according to classical laws of multiplication by scalars for determinants of matrices. Then the determinant we are interested in is equal to the determinant of the equivalent matrix that comes from performing linear combinations of rows to get $\alpha_{j-1}^{\kappa+1-i}$ on the $i^{th}$ row and $j^{th}$ column which gives a Vandermonde determinant of value $\prod_{i, j=0 | i<j}^\kappa (\alpha_i - \alpha_j)$. The determinant of $W$ directly follows by tensor calculus.
\end{proof}

Then, by Lemma \ref{lem:wronskian}, we thus find
\begin{equation*}
  \sW (s)^{-1} :=
  \frac{1}{\det \sW(s)} \operatorname{Comatrix} \sW(s).
\end{equation*}

Going back to the calculation of~\eqref{eq:first-der-k}, and combining~\eqref{eq:det-k}
with the $s$-scaling of the entries of the cofactor matrix and our choice \eqref{eq:choice-g}, we get
\begin{equation}
  \label{eq:grad-phi}
  \nabla_{(0)} \left( \Phi^s_\pm \right)^{-1}
  = \sW_\pm ^\delta (1) \big[\sW_\pm ^\delta(s)\big]^{-1}
\begin{pmatrix}
    \Id_{d_0} \\ \OO_{d_1} \\ \vdots \\ \OO_{d_\kappa}
  \end{pmatrix}
  = \sum_{i=0} ^\kappa O(s^{-1-\alpha_i}),
\end{equation}
which is integrable for $\alpha_i \in (-1,0)$. Observe also that
our choice of control functions $g_i$, $i=1,\dots,\kappa$,
implies that all their derivatives up to order $\kappa+1$ are
integrable on $s \in [0,1]$, which implies that the trajectories
are bounded (with integrable tangent vector field).

\section{Proof of the Poincaré inequality}
\label{sec:proof}

We consider the three cylinders $Q^+,Q^-, Q^0 \subset \tilde \Omega$ as
in Figure~\ref{fig:trajectories}, operators $\mathcal A$ and
$\mathcal B$ that satisfy {\bf (H)}, and a weak sub-solution $f$
of~\eqref{eq:hormander}. Let $\varphi \in C_c^\infty(\R^{N})$ be a
non-negative function in the first $N$ variables
$(x^{(\kappa)},\dots,x^{(0)})$ (excluding the time component)
with compact support in any time-slice of $Q^0$ and such that
$\fint_{Q^0} \varphi \dd x = 1$. Then
\begin{equation}
  \label{eq:poincare-aux}
  \begin{aligned}
    \fint_{Q^+} \Big(f(z_+) - \langle f\rangle_{Q^-}\Big)_+ \dd
    z_+ & \leq \fint_{Q^+} \fint_{Q^-} \Big(f(z_+) - f(z_{-})
    \Big)_+
    \dd z_{-} \dd z_+ \\
    & \leq \fint_{Q^+} \fint_{Q^-} \Big(f(z_+) - \langle
    f\varphi\rangle_{Q^0}\Big)_+\dd z_{-} \dd z_+ \\
    &\quad + \fint_{Q^+} \fint_{Q^-} \Big( \langle
    f\varphi\rangle_{Q^0} - f(z_{-})\Big)_+ \dd z_{-}\dd z_+ \\
    & \leq \fint_{Q^+} \Bigg\{\underbrace{\fint_{Q^0} \big(f(z_+)
      - f(z_0)\big)\varphi(x_0) \dd z_0}_{=: \cI^+}
    \Bigg\}_+ \dd z_+ \\
    & \quad + \fint_{Q^-} \Bigg\{\underbrace{\fint_{Q^0}
      \big(f(z_0) - f(z_{-})\big)\varphi(x_0) \dd z_0}_{=:
      \cI^-}\Bigg\}_+ \dd z_{-}.
  \end{aligned}
\end{equation}
where $z_+ = (x_+, t_+), z_- = (x_-, t_-)$ and
$z_0 = (x_0, t_0)$.  Note that $t_{-} < t_0 < t_+$. We now use
the trajectories constructed in Section \ref{sec:trajectories} to
estimate the right hand side.

Using the chain rule, \eqref{eq:hormander},
\eqref{eq:control-general} with the choice \eqref{eq:choice-g} we
get
\begin{equation}
  \label{eq:I1-I2}
  \begin{aligned}
    \cI^\pm & = \pm \fint_{Q^0} \Big(f(z_\pm) -
    f(z_0)\Big) \, \varphi(x_0) \dd z_0 \\
    & = \mp \fint_{Q^0} \int_0^1 \frac{\dd}{\dd
      s}f\big(\Gamma_\pm(s)\big) \, \varphi(x_0)\dd s\dd z_0 \\
    & = \mp\delta_\pm  \fint_{Q^0} \int_0^1 \left( \mathcal T
    f \right)\big(\Gamma_\pm(s)\big) \, \varphi(x_0)\dd s\dd z_0 \\
    & \qquad \mp \fint_{Q^0} \int_0^1 \left( \ds X^{(0)}(s) \right)
    \cdot \left(\nabla_{(0)} f\right)(\Gamma_\pm(s)) \, \varphi(x_0)\dd s \dd
    z_0 \\
    & \leq  \underbrace{\pm\delta_\pm \fint_{Q^0} \int_0^1
      \left[ \left( \mathcal A^* \mathcal A \right) f\right]
      \big(\Gamma_\pm(s)\big) \, \varphi(x_0)\dd s\dd z_0}_{=:
      \cI^\pm_1} \\
    & \qquad  \underbrace{\mp\fint_{Q^0} \int_0^1 \sum_{i =
        0}^\kappa s^{\alpha_i} \mathsf m_\pm^{(i)}
      \cdot \left( \nabla_{(0)} f \right)(\Gamma_\pm(s)) \,
      \varphi(x_0)\dd s \dd z_0}_{=: \cI^\pm_2}.
  \end{aligned}
\end{equation}
where the above integrals are interpreted in the duality sense. The only difference between the two terms $\cI^+$ and
$\cI^-$ is the role of $z_0$: in the former case it is the past
variable, in the latter it is the future variable.

In the following computations to estimate $\mathcal I_1^\pm$ and $\mathcal I_2^\pm$,
the local and the non-local case can be treated similarly,
\textit{upon replacing the integration domain over $Q^0$
by $Q^0_v \times \R^{d_0}$ in the non-local case,
where $Q^0_v \subset \R^{N + 1 - d_0}$ is the domain of the cylinder $Q^0$
for the variables $(x^{(1)}, \dots, x^{(\kappa)}, t)$,
 that is we slice out the $x^{(0)}$ variable.}

We now integrate by parts the terms $\cI^\pm _1$ after a
change of variables $x_0 \mapsto y:= \Phi^s_\pm(x_0)$ for
$s,t_0$ fixed $\Phi^s_\pm$ is the affine map defined in \eqref{eq:mapping}, which is invertible for $s
\not =0$): 
\begin{equation*}
  \begin{aligned}
    \cI_1^\pm & = \pm\frac{\delta_\pm}{|Q^0|} \int_{Q^0} \int_0^1
      \left[ \left( \mathcal A^* \mathcal A \right) f\right]
      \big(\Gamma_\pm(s)\big) \, \varphi(x_0)\dd s\dd z_0 \\
      & = \pm\frac{\delta_\pm}{|Q^0|}\int_{Q^0} \int_0^1
      \left[ \left( \mathcal A^* \mathcal A \right) f\right]
      \left( \Phi^s_\pm (x_0),st_0 + (1-s)t_\pm
      \right) \, \varphi(x_0)\dd s\dd x_0 \dd t_0 \\
      & = \pm \frac{\delta_\pm}{|Q^0|} \int_{(\Phi^s_\pm \otimes \I)(Q^0)} \int_0^1
      \left[ \left( \mathcal A^* \mathcal A \right) f\right]
      \left( y,st_0 + (1-s)t_\pm \right) \, \varphi
      \left( \left(\Phi^s_\pm\right)^{-1}(y)\right) \frac{\dd s
        \dd y \dd t_0}{\left| \det \mathfrak
        A^s \right|} \\
      & = \pm \frac{\delta_\pm}{|Q^0|} \int_{(\Phi^s_\pm \otimes
        \I)(Q^0)} \int_0^1 \left[ \mathcal A f\right]
      \left( y,st_0 + (1-s)t_\pm \right) \, \cA\left[\varphi
      \left(\left( \Phi^s_\pm \right)^{-1} (y)\right)\right]
    \frac{\dd s \dd y \dd t_0}{\left| \det \mathfrak A^s \right|} \\
      & = \pm \frac{\delta_\pm}{|Q^0|}  \int_{Q^0} \int_0^1
      \left[ \mathcal A f\right]
      \left( \Gamma_\pm(s) \right) \, \left\{ \cA\left[\varphi
      \left(\left( \Phi^s_\pm \right)^{-1} (y)\right)\right]
  \right\}_{|y=\Phi^s_\pm(x_0)} \dd s \dd x_0 \dd t_0.
  \end{aligned}
\end{equation*}

We then use the first bound in \eqref{eq:hypA} on the operator $\mathcal A$ in
assumption {\bf (H)}: for $\epsilon \in (0, 1-\beta)$ 
\begin{align*}
  \left| \left\{ \cA\left[\varphi
  \left(\left( \Phi^s_\pm \right)^{-1} (y)\right)\right]
  \right\}_{|y=\Phi^s_\pm(x_0)} \right|
  & \lesssim_\varphi
    \left\| \nabla_{(0)} \left( \Phi^s_\pm \right)^{-1}
    \right\|_\infty ^{\beta +\epsilon}\\
  & \lesssim_\varphi \left( \sum_{i=0} ^\kappa s^{-1-\alpha_i}
    \right)^{\beta +\epsilon} \lesssim_\varphi \left( \sum_{i=0} ^\kappa
    s^{-(\beta+\epsilon)(1+\alpha_i)} \right)
\end{align*}
and deduce finally
\begin{equation*}
  \cI_1^\pm \lesssim \frac{\delta_\pm}{|Q^0|} \int_{Q^0} \int_0^1
  \left| \left[ \mathcal A f\right]
    \left( \Gamma_\pm(s) \right) \right|  \left( \sum_{i=0}
    ^\kappa s^{-(\beta+\epsilon)(1+\alpha_i)} \right) \dd s \dd x_0 \dd
  t_0.
\end{equation*}

Now let us turn to $\cI_2^\pm$ in \eqref{eq:I1-I2}. We
find, using the same change of variables,
\begin{equation*}
  \begin{aligned}
    \cI_2^\pm
    & = \mp\int_0^1  \fint_{Q^0}  \sum_{i = 0}^\kappa
    s^{\alpha_i} \mathsf m_\pm^{(i)} \cdot\nabla_{(0)}
    f(\Gamma_\pm(s)) \varphi(x_0) \dd s \dd x_0 \dd t_0 \\
    & = \mp \frac{1}{\vert Q^0\vert} \sum_{i = 0}^\kappa
    \int_0^1  \int_{(\Phi^s_\pm \otimes \I)(Q^0)}  s^{\alpha_i}
    \mathsf m_\pm^{(i)}  \cdot\nabla_{(0)}
    f \left( y,st_0 + (1-s)t_\pm \right) \\
    & \hspace{6cm} \times \varphi
    \left(\left( \Phi^s_\pm \right)^{-1} (y)\right)
    \frac{\dd s \dd y \dd t_0}{\left| \det \mathfrak A^s \right|}\\
      & ={\pm}  \frac{1}{\vert Q^0\vert }\sum_{i = 0}^\kappa
      \int_0^1   \int_{{Q^0}} s^{\alpha_i} \left( (-\Delta_{(0)})^{\frac{\beta}{2}}
       f\right)\left( \Gamma_\pm(s) \right) \\
      & \hspace{2cm} \times \mathsf m_\pm^{(i)}
       \cdot \left\{  (-\Delta_{(0)})^{-\frac{\beta}{2}}\nabla_{(0)}
      \left\{\varphi\left[ \left( \Phi_\pm^s \right)^{-1} (y)\right]\right\}
    \right\}_{|y=\Phi^s_\pm(x_0)} \dd s \dd x_0 \dd t_0.
\end{aligned}
\end{equation*}
{
For $0<\beta<1$, we claim that
\begin{equation}
\label{e.scaling.aux}
\begin{aligned}
  \left|  (-\Delta_{(0)})^{-\frac{\beta}{2}}\nabla_{(0)}
    \left\{\varphi\left[ \left( \Phi_\pm^s \right)^{-1}
        (y)\right]\right\} \right| 
        &\lesssim
  \left\| \nabla_{(0)} \left( \Phi^s _\pm \right)^{-1}
  \right\|_\infty ^{1-\beta} 
  \\
  &
  \lesssim  \left(
    \sum_{j=0} ^\kappa s^{-1-\alpha_j} \right)^{1-\beta}
   \leq 
    \sum_{j=0} ^\kappa s^{-(1+\alpha_j)(1-\beta)} .
\end{aligned}
\end{equation}
This can be seen as follows. For \(0<\beta<1\), the operator
\[
T_\beta:=(-\Delta_{(0)})^{-\beta/2}\nabla_{(0)}
\]
has order \(1-\beta\) and kernel
\[
K_\beta(v-w)
 =c_{d_0,\beta}\frac{v-w}{|v-w|^{d_0+2-\beta}}.
\]
Writing $y=(y',v)$, where
$y'=(y^{(\kappa)},\ldots,y^{(1)})$ is fixed, we set
\[
	g(v):=\varphi\big((\Phi_\pm^s)^{-1}(y',v)\big)
\]
so that, using the cancellation of this odd kernel, 
\[
|T_\beta g(v)|
 \lesssim
 \int_{\mathbb R^{d_0}}
 \frac{|g(v)-g(w)|}{|v-w|^{d_0+1-\beta}}\,dw .
\]
By splitting the integral at radius \(\rho\), 
\[
|T_\beta g(v)|
 \lesssim \|\nabla_{(0)}g\|_\infty\rho^\beta+\|g\|_\infty\rho^{\beta-1}.
\]
Choosing \(\rho=\frac{\|g\|_\infty}{\|\nabla_{(0)}g\|_\infty}\) gives
\[
\|T_\beta g\|_\infty
 \lesssim_{\beta,d_0}
 \|g\|_\infty^\beta
 \|\nabla_{(0)}g\|_\infty^{1-\beta}.
\]
By the chain rule,
\[
	\|\varphi\big((\Phi_\pm^s)^{-1}(y)\big)\|_\infty\leq \|\varphi\|_\infty,
	\qquad
	\|\nabla_{(0)}\Big(\varphi\big((\Phi_\pm^s)^{-1}(y)\big)\Big)\|_\infty
 	\lesssim_\varphi
	 \|\nabla_{(0)}(\Phi_\pm^s)^{-1}\|_\infty.
\]
Thus, combined with \eqref{eq:grad-phi}, this implies \eqref{e.scaling.aux}.
}

We deduce
\begin{equation*}
  \cI_2^\pm
  \le \frac{1}{\vert Q^0\vert }\int_0^1 \int_{Q^0} \left( \sum_{i,j=0} ^\kappa s^{ \alpha_i +\beta(1+\alpha_j) - \alpha_j - 1} \right) \left|
  \left((-\Delta_{(0)})^{\frac{\beta}{2}} f\right)\left(
      \Gamma_\pm(s) \right) \right| \dd s \dd x_0 \dd t_0.
\end{equation*}

It follows from \eqref{eq:poincare-aux} and \eqref{eq:I1-I2} that
we are left with estimating
{
\begin{align*}
\cJ := \cJ^+ + \cJ^-, \quad \text{where} \quad
  \cJ^{\pm} := \fint_{Q^\pm} \big\{\cI^\pm_1\big\}_+ \dd
  z_\pm + \fint_{Q^\pm}  \big\{\cI^\pm_2\big\}_+ \dd
  z_\pm.
\end{align*}
}
The previous estimates imply
\begin{equation}
  \label{eq:J-gen}
  \begin{aligned}
    {\cJ^{\pm}} & \lesssim \int_{Q^\pm} \int_0^1 \int_{Q^0}
    \left( \sum_{i = 0}^\kappa \frac{1}{ s^{(\alpha_i+1)(\beta+\epsilon)}}
    \right) \left| \left( \cA f \right) (\Gamma_\pm(s))\right|
    \dd z_\pm \dd s
    \dd z_0 \\
    & +  \int_{Q^\pm} \int_0^1\int_{Q^0}  \left( \sum_{i,j =
        0}^\kappa s^{\alpha_i +\beta(1+\alpha_j) - \alpha_j -1} \right) \left| \left(
        (-\Delta_{(0)})^{\frac{\beta}{2}}f\right)\left(
        \Gamma_\pm(s) \right) \right| \dd z_\pm \dd s \dd z_0.
  \end{aligned}
\end{equation}
Given some $\var>0$, we first choose $\epsilon>0$ sufficiently small
and then choose the $\alpha_i$'s pairwise distinct and sufficiently
close to $(2+\beta+\epsilon)^{-1}-1$.


\begin{remark}
\label{rmk_gain_int}
Note also that, by adopting the mollified
formulation~\eqref{eq:moll-Poincare} and by keeping the test
functions $\psi_\pm$ until the integral {$\cJ^{\pm}$} above, one could
use the variable $z_0$ to integrate the test function in the
non-singular region $s \in [s_0,1]$, keep an $L ^\infty$ control
of the test function in the singular region, and optimise
$s_0$ in order to deduce a gain of integrability.
\end{remark}

Then~\eqref{eq:J-gen} simplifies into, for
some $\var$ as small as wanted,
\begin{align*}
  {\cJ^{\pm}} &\lesssim \int_{Q^\pm} \int_0^1\int_{Q^0}
    s^{-\frac{\beta}{2+\beta}-\var} 
    	\left| \left( \cA f \right) (\Gamma_\pm(s))\right|  \dd z_\pm \dd s \dd z_0 \\
    &+ \int_{Q^\pm} \int_0^1\int_{Q^0} s^{-1 + \frac{\beta}{2+\beta}-\var}  \left|\left((-\Delta_{(0)})^{\frac{\beta}{2}} 
           f\right)\left(\Gamma_\pm(s) \right) \right|  \dd z_\pm \dd s \dd z_0.
\end{align*}
We are now in a position to use not only the intermediate
variable $z_0$, but also the future/past variables $z_\pm$ for a
change of variables, since we are now integrating in both $Q^0$
and $Q_\pm$. Note that it was not possible to use the integration
in $Q_\pm$ before because of the positive value around the $Q^0$
integral.

We split {$\cJ^{\pm}$} as follows, given $s_0 \in (0,1)$,
\begin{equation*}
  \begin{aligned}
     \int_{Q^\pm} \int_{Q^0} &\int_0^1
     \Big( \cdots \Big) \dd s \dd z_0 \dd z_\pm
    \\
    & = \underbrace{\int_{Q^\pm} \int_{Q^0} \int_0^{s_0} \Big( \cdots \Big)  \dd s \dd z_0 \dd
      z_\pm}_{=: \cJ^\pm_1} + \underbrace{\int_{Q^\pm}
      \int_{Q^0} \int_{s_0}^1 \Big(
      \cdots \Big) \dd s \dd z_0 \dd z_\pm}_{=: \cJ^\pm_2}.
  \end{aligned}
\end{equation*}
The two changes of variables on each part are represented in
Figure~\ref{fig:change-of-var}.

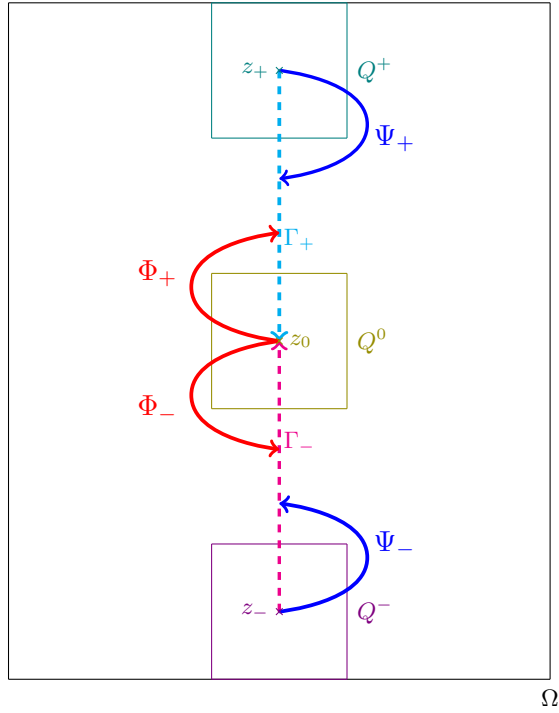
\begin{figure}{
    \centering
    \begin{tikzpicture}[scale =0.9]
      \draw [black] (-4,0) -- (4,0);
      \draw[black] (-4,-10) -- (4,-10);
      \draw [black](-4,0) -- (-4,-10);
      \draw [black](4,0) -- (4,-10) node[anchor=north, scale=1]
      {\footnotesize{$\tilde \Omega$}};
      \draw [teal](-1,0) -- (1, 0);
      \draw [teal](-1,-2) -- (1, -2);
      \draw [teal](-1,0) -- (-1, -2);
      \draw [teal](1,0) -- node[anchor=west, scale=1]
      {\footnotesize{$Q^+$}}(1, -2);
      \draw[olive] (-1,-4) -- (1, -4);
      \draw [olive](-1,-6) -- (1, -6);
      \draw [olive](-1,-4) -- (-1, -6);
      \draw [olive](1,-4) -- node[anchor=west, scale=1]
      {\footnotesize{$Q^0$}}(1, -6);
      \draw[violet] (-1,-8) -- (1, -8);
      \draw [violet](-1,-10) -- (1, -10);
      \draw [violet](-1,-8) -- (-1, -10);
      \draw [violet](1,-8)  -- node[anchor=west, scale=1]
      {\footnotesize{$Q^-$}} (1, -10);
      \draw [->, cyan, dashed, line width=0.5mm]  plot [smooth]
      coordinates {(0,-1)(0,-5)};
      \node [cyan] at (0.3, -3.5) {\footnotesize{$\Gamma_+$}};
      \draw [->, magenta, dashed, line width=0.45mm] plot
      [smooth] coordinates {(0,-9) (0,-5)};
      \node [magenta] at (0.3, -6.5) {\footnotesize{$\Gamma_-$}};
      \draw [->, blue, line width=0.45mm] plot [smooth, tension =
      1.5] coordinates {(0,-1) (1.3, -1.8)(0,-2.6)};
      \node [blue] at (1.7, -2) {$\Psi_+$};
      \draw [->, blue, line width=0.5mm] plot [smooth, tension =
      1.5] coordinates {(0,-9) (1.3, -8.2) (0,-7.4)};
      \node [blue] at (1.7, -8) {$\Psi_-$};
      \draw [->, red, line width=0.5mm] plot [smooth, tension =
      1.5] coordinates {(0,-5) (-1.3, -4.2)(0,-3.4)};
      \node [red] at (-1.8, -4) {$\Phi_+$};
      \draw [->, red, line width=0.5mm] plot [smooth, tension =
      1.5] coordinates {(0,-5) (-1.3, -5.8) (0,-6.6)};
      \node [red] at (-1.8, -6) {$\Phi_-$};
      \coordinate[label={[teal]left:\footnotesize{$z_+$}}] (z_+) at (0, -1);
      \draw [teal] plot[only marks,mark=x,mark size=2pt]  coordinates {(0, -1)};
      \coordinate[label={[olive]right:\footnotesize{$z_0$}}] (z_0) at (0, -5);
      \draw [olive] plot[only marks,mark=x,mark size=2pt]
      coordinates {(0, -5)};
      \coordinate[label={[violet]left:\footnotesize{$z_-$}}] (z_-) at (0, -9);
      \draw [violet] plot[only marks,mark=x,mark size=2pt]
      coordinates {(0, -9)};
    \end{tikzpicture}
  }
  \caption{The change of variables that we use in the proof for
    some fixed $s \in (0, 1)$. For $s \in (0,s_0)$ we use
    $\Psi_\pm$ that map $z_\pm$ onto $ \Gamma_\pm$, whereas
    for $s \in (s_0,1)$ we use $\Phi_\pm$ that map $z_0$ onto
    $\Gamma_\pm$.}\label{fig:change-of-var}
\end{figure}

To control the part $\cJ^\pm_1$ we parametrise
$\Gamma_\pm(s)$ by the $z_\pm$ coordinates, for $s \in [0,s_0]$
and $t_0, t_\pm$ all fixed:
\begin{equation*}
  (x_\pm,t_\pm) \to \Gamma_\pm (s) = (\bar \Gamma^s_\pm,t_0 s +
  (1-s) t_\pm) =: {\Psi^s_\pm(x_\pm,t_\pm)} 
\end{equation*}
Since $s$ does not approach $1$, we can prove that this change of
variables is not singular. The mappings $\Psi_\pm^s$ are
determined by solving~\eqref{eq:mapping} for $x_\pm$ instead of
$x_0$. It yields
\[
\Psi_\pm^s (x_\pm,t_\pm)= 
(\mathfrak a_\pm^s x_\pm + \mathfrak b_\pm^s,t_0 s + (1-s) t_\pm)
\]
with the matrix $\mathfrak a_\pm^s$ and vector $\mathfrak b_\pm^s$ given
by
\begin{equation*}
  \mathfrak a^s _\pm :=  \left( \sT_{\pm} (s) - \sW_\pm ^\delta (s) \big[\sW_\pm ^\delta
    (1)\big]^{-1} \sT_{\pm} (1) \right)  
    \quad
  \mathfrak b^s _\pm := \sW_\pm ^\delta (s) \big[\sW_\pm ^\delta
    (1)\big]^{-1} x_0,
\end{equation*}
which depend only on $s$ and $x_0$. Since we have proved in
Lemma~\ref{lem:wronskian} that
\[
\sW(s) = \cO\left(s^{d_0
  \frac{(\kappa+1)(\kappa+2)}{2}+d_0\sum_{i=0}^\kappa \alpha_i}\right)
  \]
it follows that
$\sW^\delta_\pm(s) = \cO\left(s^{d_0
  \frac{(\kappa+1)(\kappa+2)}{2}+d_0\sum_{i=0}^\kappa \alpha_i}\right)$
and thus $\sW^\delta_\pm(s)$ goes to zero as $s \to 0$. Since
$\sT^\delta_\pm$ is a lower triangular matrix with unitary
diagonal it is invertible for all $s$, and finally
$\mathfrak a^s_\pm$ is invertible with uniformly bounded inverse
on $s \in [0,s_0]$ for $s_0>0$ small enough. We apply this change
of variables to estimate $\cJ_1^\pm$:
\begin{align*}
  \cJ^\pm _1
&  = \int_{Q^0} \int_0 ^{s_0} {\int_{\Psi^{-1}(Q^\pm)}} s^{-\frac{\beta}{2+\beta}-\var} 
  \left| \left(\cA f\right)\left(
      y, t\right) \right| \frac{\dd y \dd t \dd s \dd z_0}{\left| {(1-s)} \det \mathfrak a^s \right|}\\
 & +   \int_{Q^0} \int_0 ^{s_0} {\int_{\Psi^{-1}(Q^\pm)}} s^{-1 + \frac{\beta}{2+\beta}-\var}
  \left| \left(    (-\Delta_{(0)})^{\frac{\beta}{2}} f\right)\left(
      y, t\right) \right| \frac{\dd y \dd t \dd s \dd z_0}{\left|{(1-s)}  
      \det \mathfrak a^s \right|}\\
  & \lesssim_{s_0} \int_{\tilde\Omega}  \left|
    \cA f(z) \right| \dd z,
\end{align*}
{where we also applied Fubini-Tonelli's theorem and integrated over $s$.}
The part $\cJ^\pm_2$ is controlled by using the same
change of variables $\Phi^s_\pm$ as before, {this time also involving the time component, and for this reason we call it $\tilde{\Phi}^s_\pm(x_0, t_0) := (\Phi^s_\pm(x_0), t_0 s +(1-s) t_\pm)$:}
\begin{align*}
  \cJ^\pm_2
  & = \int_{Q^\pm} \int_{s_0}^1 {\int_{\tilde{\Phi^s_\pm}^{-1}(Q^0)}}  s^{-1 + \frac{\beta}{2+\beta}-\var}
    \left| \left(
    (-\Delta_{(0)})^{\frac{\beta}{2}}f\right)\left(
    y, t\right) \right| \frac{\dd y \dd t \dd s \dd z_\pm}{\left| {s} \det \mathfrak A^s \right|} \\
    & + \int_{Q^\pm} \int_{s_0}^1 {\int_{\tilde{\Phi^s_\pm}^{-1}(Q^0)}}  s^{-\frac{\beta}{2+\beta}-\var} 
    \left| \cA f \left(
    y, t \right) \right| \frac{\dd y \dd t \dd s \dd z_\pm}{\left|  {s} \det \mathfrak A^s \right|} \\
   &\lesssim
    \int_{\tilde \Omega}  \left| \cA f(z) \right| \dd z,
\end{align*}
where we have used first that the integral in $s$ avoids the
singularity at $0$, and second assumption~\eqref{eq:hypA} in
{\bf (H)}. This concludes the proof.

\section{Proof of the weak Harnack inequality}
\label{sec:DGlemmas}

In this section we briefly explain how to adapt the approach
of~\cite{guerand-mouhot} for deducing the intermediate value of
lemma of De Giorgi from the Poincaré inequality we have proved. This
approach was developed in the local case with $\kappa = 1$ and
with a weaker version of the Poincaré inequality. A non-local
extension, when $\kappa = 1$, was developed
in~\cite{loher-2022-preprint-quantitative-de-giorgi}.
This approach can be described as a ``trajectory viewpoint'' on
the De Giorgi theory, and it is fully quantitative. 

In order to carry out this analysis, one needs to have an a priori local quantitative boundedness result
for weak solutions to \eqref{eq:hormander}, also known as the first De
Giorgi lemma, which writes in its simplest form as follows.
\begin{lemma}[First Lemma of De Giorgi]
  Let the operators $\cA, \cB$ be such that $\cB$ is given by
  {\bf (H)} and $\cA$ is more specifically given by
  \begin{itemize}
  \item[(i)] $\cA:=\sqrt{A} \nabla_{(0)}$ with
    $A \in [\Lambda^{-1},\Lambda]$ measurable symmetric matrix (local
    case),
  \item[(ii)] $\cA = \sqrt{a} (-\Delta_v)^{\beta/2}$ with
    $a \in [\Lambda^{-1},\Lambda]$ scalar measurable (non-local
    case).
  \end{itemize}
  Let $f$ be a sub-solution to~\eqref{eq:hormander} in $Q_1$, such that in the local case
  $0 \le f$ in $Q_1$, and such that in the non-local case, $0 \le f \le 1$ in 
  $B_1^{\otimes \kappa}\times \R^d \times [-1, 0]$. Then there exist $\var>0$ and
  $r \in (0,1)$ so that $\int_{Q_1} f^2 \le \var$ implies
  $f \le 1/2$ on $Q_r$.
\end{lemma}

In the local case, a stronger statement holds where the
$L^\infty$ norm of $f$ on $Q_r$ is controlled by a constant
(depending on $\cA$, $\cB$ and $r$) times the $L^2$ norm of $f$
on $Q_1$. In the non-local case, variants of this stronger
statement exist if $f$ is not only a sub-solution but also a
super-solution; we refer to~\cite{Loher-2023, Loher-2024}. The proof is based
on an energy estimate and the gain of integrability of the
fundamental solution of the Kolmogorov equation, local or
non-local. This is classical and we refer
to~\cite{guerand-mouhot,loher-2022-preprint-quantitative-de-giorgi,AR-harnack, PP, IS-weak, APP}. The
non-local case with $\kappa \geq 2$ is obtained
by scaling arguments, or by calculating the fundamental
solutions in Fourier. 

As a next step, one needs to prove the second De Giorgi Lemma, also known as the intermediate value
Theorem, which controls the gradient of a sub-solution by  getting an explicit bound on the measures of the sets where the solution is in between two values. 
It reads as follows:
\begin{lemma}[Second Lemma of De Giorgi]
  Let the operators $\cA, \cB$ be such that $\cB$ is given by
  {\bf (H)} and $\cA$ is more specifically given by
  \begin{itemize}
  \item[(i)] $\cA:=\sqrt{A} \nabla_{(0)}$ with
    $A \in [\Lambda^{-1},\Lambda]$ measurable symmetric matrix (local
    case),
  \item[(ii)] $\cA = \sqrt{a} (-\Delta_v)^{\beta/2}$ with
    $a \in [\Lambda^{-1},\Lambda]$ scalar measurable (non-local
    case).
  \end{itemize}
  Let $0 \leq f \leq 1$ be a sub-solution to~\eqref{eq:hormander} in $Q_1$. Let $\delta_1, \delta_2 \in (0, 1)$. Then there are
  constants $r_0 > 0$ and $\nu \in (0, 1)$ and
  $\theta \in (0, 1)$ such that whenever
  \begin{equation*}
    \label{eq:conditions-ivl}
    \vert \{f \leq 0\} \cap Q_{r_0}^-\vert \geq \delta_1 \vert
    Q_{r_0}^-\vert \qquad \textrm{ and } \qquad \vert \{f \geq
    1-\theta\} \cap Q_{r_0}\vert \geq \delta_2 \vert Q_{r_0}\vert,
  \end{equation*}
  there holds
  \begin{equation*}
    \big\vert \{0 < f < 1-\theta\} \cap Q_{1/2}\big\vert \geq \nu
    \vert  Q_{1/2}\vert
  \end{equation*}
  in the local case, or
  \begin{equation*}
    \big\vert \{0 < f < 1-\theta\} \cap
    B_{(\frac12)^{1+\kappa\beta}} \times \cdots \times
    B_{(\frac12)^{1+\beta}} \times [-3,0] \big\vert \geq \nu
    \vert  Q_{1/2}\vert
  \end{equation*}
  in the non-local case.
\end{lemma}

The proof of this lemma follows from our Poincaré inequality
and the first lemma of De Giorgi by using the argument
of~\cite[Subsection~3.2]{guerand-mouhot} in the local case (in
fact the proof is simpler since we do not have the error terms in
the Poincaré inequality that were present
in~\cite{guerand-mouhot}), and it follows
from~\cite[Subsection~5.2]{loher-2022-preprint-quantitative-de-giorgi}
in the non-local case (note that we have removed the
mollifications $F_i$ of the cutoff functions in the statement
of~\cite{loher-2022-preprint-quantitative-de-giorgi} since they
can be used in the proof but removed from the statement by making the border region small
enough). These proofs are done in the case of one commutator but
they extend straightforwardly to the case of several commutators
because they do not depend on the transport part
$\cT = \partial_t + \cB$ of the equation, but only on $\cA$ and
the Poincaré inequality and first lemma of De Giorgi.

With the first and second De Giorgi lemma at hand, one can deduce
the ``measure-to-pointwise estimate'' and the weak Harnack
inequality by following~\cite{guerand-mouhot} in the local case,
and~\cite{loher-2022-preprint-quantitative-de-giorgi} in the
non-local case, see \cite{APP} also. These proofs are independent of the equation and
only depend on the previously established functional
inequalities. The only minor change is that the kinetic cylinders
have a slightly more complicated scaling on the variables
$(x^{(\kappa)},\dots,x^{(0)},t)$ but the technical changes needed
are, although slightly tedious, clear.


\begin{thebibliography}{30}



\bibitem{APP}
{\sc Anceschi, F., Palatucci, G., and Piccinini, M.}
\newblock De Giorgi-Nash-Moser theory for kinetic equations with nonlocal diffusions.
\newblock {\em Preprint} (2024),
  ArXiv: 2401.14182v4.

\bibitem{AR-harnack}
{\sc Anceschi, F., and Rebucci, A.}
\newblock A note on the weak regularity theory for degenerate {K}olmogorov
  equations.
\newblock {\em J. Differ. Equ. 341\/} (2022), 538--588.



\bibitem{Barata-Hussein}
	{\sc Barata, J., and Hussein, M.}
\newblock The {M}oore--{P}enrose {P}seudoinverse: {A} {T}utorial {R}eview of the {T}heory.
\newblock {\em Brazilian Journal of Physics 42\/} (2012), 146--165.


\bibitem{BP}
{\sc Barraquand, J., and Pudet, T.}
\newblock {P}ricing of {A}merican path-dependent contingent claims.
\newblock {\em Math. Finance 6\/} (1996), 17--51.



\bibitem{MR553920}
{\sc Bogovski\u{\i}, M.}
\newblock Solution of the first boundary value problem for an equation of
  continuity of an incompressible medium.
\newblock {\em Dokl. Akad. Nauk SSSR 248}, 5 (1979), 1037--1040.

\bibitem{DHW}
{\sc Dewynne, J., Howison, S., and Wilmott, P.}
\newblock {\em Option pricing}.
\newblock Oxford Financial Press. Oxford, 1993.

\bibitem{dietert2023quantitative}
{\sc Dietert, H., Hérau, F., Hutridurga, H., and Mouhot, C.}
\newblock Quantitative geometric control in linear kinetic theory, 
\newblock {{\em Preprint}} (2022) {ArXiv: 2209.09340}.

\bibitem{GIMV}
{\sc Golse, F., Imbert, C., Mouhot, C., and Vasseur, A.~F.}
\newblock {Harnack inequality for kinetic Fokker-Planck equations with rough
  coefficients and application to Landau equation}.
\newblock {\em Ann. Sc. Norm. Super. Pisa Cl. Sci. 19}, 5 (2019), 253--295.

\bibitem{guerand-mouhot}
{\sc Guerand, J., and Mouhot, C.}
\newblock Quantitative {D}e {G}iorgi methods in kinetic theory.
\newblock {\em J. \'{E}c. polytech. Math. 9\/} (2022), 1159--1181.

\bibitem{HPZ}
{\sc Hao, Z., Peng, X., and Zhang, X.}
\newblock Hörmander’s Hypoelliptic Theorem for Nonlocal Operators. 
\newblock {\em J. Theor. Probability} 34 (2021), 1870--1916.

\bibitem{MR0222474}
{\sc H{\" o}rmander, L.}
\newblock Hypoelliptic second order differential equations.
\newblock {\em Acta Math. 119\/} (1967), 147--171.


\bibitem{IS-weak}
{\sc Imbert, C. and Silvestre, L.}
\newblock The weak Harnack inequality for the Boltzmann equation without cut-off.  
\newblock {\em J. Eur. Math. Soc. (JEMS)} 22 (2020), no. 2, 507--592


\bibitem{kolmogorov}
{\sc Kolmogoroff, A.}
\newblock Zuf{\"a}llige {Bewegungen}. ({Zur} {Theorie} der {Brownschen}
  {Bewegung}.).
\newblock {\em Ann. Math. (2) 35\/} (1934), 116--117.

\bibitem{Loher-2023}
{\sc Loher, A.}
\newblock The strong {H}arnack inequality for the {B}oltzmann equation.
\newblock In {\em Séminaire Laurent Schwartz - EDP et applications\/}
  (2023-2024), Centre Mersenne, pp.~1--15.

\bibitem{loher-2022-preprint-quantitative-de-giorgi}
{\sc Loher, A.}
\newblock {Quantitative De Giorgi methods in kinetic theory for non-local
  operators}.
\newblock {\em J. Funct. Anal. 286}, 6 (2024).


\bibitem{Loher-2024}
{\sc Loher, A.}
\newblock {Semi-local behaviour of non-local hypoelliptic equations: divergence form}.
\newblock {\em arXiv:2404.05612} (2024).



\bibitem{moser}
{\sc Moser, J.}
\newblock A {H}arnack inequality for parabolic differential equations.
\newblock {\em Comm. Pure Appl. Math. 17\/} (1964), 101--134.

\bibitem{nash}
{\sc Nash, J.}
\newblock Continuity of solutions of parabolic and elliptic equations.
\newblock {\em Amer. J. Math. 80\/} (1958), 931--954.

\bibitem{niebel2022kinetic}
{\sc Niebel, L., and Zacher, R.}
\newblock On a kinetic {P}oincar\'e inequality and beyond.
\newblock J. Func. Anal. 289 (1) (2025), 110899.

\bibitem{Pascucci}
{\sc Pascucci, A.}
\newblock {\em PDE and {M}artingale methods in option pricing}.
\newblock Milano, Springer. Bocconi and Springer Series, 2011.

\bibitem{PP}
{\sc Pascucci, A., and Polidoro, S.}
\newblock The {M}oser's iterative method for a class of ultraparabolic
  equations.
\newblock {\em Commun. Contemp. Math. 6}, 3 (2004), 395--417.

\end{thebibliography}
\end{document}